\documentclass[11pt]{article}
\usepackage{latexsym,amssymb,amsmath,amsfonts,amsthm}
\usepackage{graphics}
\usepackage{graphicx}
\usepackage{mathrsfs}
\usepackage{subfigure}
\newcommand{\R}{{\mat R}}

\newcommand{\N}{{\mat N}}

\newcommand{\Sp}{{\mat S}}
\newcommand{\ds}{\displaystyle}
\newcommand{\no}{\nonumber}
\newcommand{\be}{\begin{eqnarray}}
\newcommand{\ben}{\begin{eqnarray*}}
\newcommand{\en}{\end{eqnarray}}
\newcommand{\enn}{\end{eqnarray*}}

\newcommand{\pa}{\partial}

\newcommand{\ov}{\overline}
\newcommand{\I}{{\rm Im}}
\newcommand{\Rt}{{\rm Re}}
\newcommand{\curl}{{\rm curl\,}}

\newcommand{\dive}{{\rm div\,}}
\newcommand{\g}{\gamma}
\newcommand{\G}{\Gamma}

\newcommand{\vep}{\varepsilon}
\newcommand{\Om}{\Omega}
\newcommand{\om}{\omega}
\newcommand{\sig}{\sigma}

\newcommand{\al}{\alpha}

\newcommand{\wi}{\widehat}
\newcommand{\ti}{\times}
\newcommand{\wid}{\widetilde}

\newcommand{\mat}{\mathbb}
\newcommand{\se}{\setminus}
\newcommand{\ify}{\infty}

\newcommand{\la}{\lambda}

\newcommand{\tr}{\triangle}
\newtheorem{theorem}{Theorem}[section]
\newtheorem{lemma}[theorem]{Lemma}
\newtheorem{corollary}[theorem]{Corollary}
\newtheorem{definition}[theorem]{Definition}
\newtheorem{remark}[theorem]{Remark}

\begin{document}
\renewcommand{\theequation}{\arabic{section}.\arabic{equation}}
\begin{titlepage}
\title{\bf
Uniqueness in inverse acoustic and electromagnetic scattering by penetrable obstacles with embedded objects
}
\author{Jiaqing Yang\thanks{School of Mathematics and Statistics, Xi'an Jiaotong University,
Xi'an, Shaanxi, 710049, China ({\tt jiaq.yang@mail.xjtu.edu.cn,})}
\and Bo Zhang\thanks{LSEC and Academy of Mathematics and Systems Sciences, Chinese Academy of Sciences,
Beijing, 100190, China and School of Mathematical Sciences, University of Chinese Academy of Sciences,
Beijing 100049, China ({\tt b.zhang@amt.ac.cn})}
\and Haiwen Zhang\thanks{Academy of Mathematics and Systems Science, Chinese Academy of Sciences,
Beijing 100190, China ({\tt zhanghaiwen@amss.ac.cn})}}
\date{}
\end{titlepage}
\maketitle

\vspace{.2in}

\begin{abstract}
This paper considers the inverse problem of scattering of time-harmonic acoustic and
electromagnetic plane waves by a bounded, inhomogeneous, penetrable obstacle with embedded objects inside.
A new method is proposed to prove that the inhomogeneous penetrable obstacle
can be uniquely determined from the far-field pattern at a fixed frequency, disregarding its contents.
Our method is based on constructing a well-posed interior transmission problem in a {\em small} domain
associated with the Helmholtz or modified Helmholtz equation and the Maxwell or modified Maxwell equations.
A key role is played by the smallness of the domain which ensures that the lowest transmission eigenvalue
is large so that a given wave number $k$ is not an eigenvalue of the interior transmission problem.
Another ingredient in our proofs is a priori estimates of solutions to the transmission scattering problems
with data in $L^p$ ($1<p<2$), which are established in this paper by using the integral equation method.
A main feature of the new method is that it can deal with the acoustic and electromagnetic cases
in a unified way and can be easily applied to deal with inverse scattering by unbounded rough interfaces.

\vspace{.2in}
\end{abstract}

\section{Introduction}
\setcounter{equation}{0}

Consider the problem of scattering of a time-harmonic acoustic or electromagnetic wave by an inhomogeneous
penetrable obstacle containing possibly some embedded objects and surrounded by a homogeneous background medium.
This problem occurs in various applications such as radar and sonar, remote sensing, geophysics, medical
imaging and nondestructive testing.

Denote by $D$ the bounded penetrable obstacle in $\R^3$ which may contain certain embedded objects
denoted by $D_b$, that is, $D_b\Subset D$. Here, $D$ is assumed to be an open bounded domain
in $\R^3$ with a smooth boundary $\pa D\in C^2$ and $D_b=\bigcup_{j=1}^m D_b^{(j)}$
with $D_b^{(j_1)}\cap D_b^{(j_2)}=\emptyset$ if $j_1\not= j_2$.
Denote by $n(x)$ the refractive index characterizing the inhomogeneous medium in $\R^3\se\ov{D_b}$ and
assume that $n=1$ in $\R^3\se\ov{D}$, $n\in L^\infty(D\se\ov{D}_b)$, $\Rt(n)>0$ and $\I(n)\geq 0$
in $D\se\ov{D}_b$. Then the scattering problem of a time-harmonic plane acoustic wave is modeled by
the transmission problem
\be\label{1.1}
\Delta u+k^2u=0 && \mbox{in}\;\;\R^3\se\ov{D},\\ \label{1.2}
\Delta v+k^2nv=0 && \mbox{in}\;\;D\se\ov{D}_b,\\ \label{1.3}
u-v=0,\quad\frac{\pa u}{\pa\nu}-\la\frac{\pa v}{\pa\nu}=0 &&\mbox{on}\;\;\pa D,\\ \label{1.4a}
\mathscr{B} (v)=0 &&\mbox{on}\;\;\pa D_b,\\ \label{1.4}
\lim_{r\rightarrow\infty}r\Big(\frac{\pa u^s}{\pa r}-iku^s\Big)=0 && r=|x|,
\en
where $\la>0$ is the transmission coefficient depending on the properties of the media
in $D\se\ov{D}_b$ and $\R^3\se\ov{D}$, $u:=u^i+u^s$ denotes the total field in $\R^3\se\ov{D}$,
$u^i=e^{ikx\cdot d}$ is the incident plane wave, $u^s$ is the scattered wave, and $\nu$ is
the unit normal on $\pa D$ directed into the exterior of $D$.
Here, the wave number $k>0$ is given by $k=\omega/c$ with the frequency $\omega>0$ and the sound
speed $c>0$ and $d\in\Sp^2$ is the incident direction.
$\mathscr{B}$ denotes the boundary condition imposed on the boundary $\pa D_b$ of the embedded
obstacle $D_b$ satisfying that $\mathscr{B}(v):=v$ if $D_b$ is a sound-soft obstacle,
and $\mathscr{B}v:=\pa_\nu v$ with the outward normal $\nu$ directing into $D\se\ov{D_b}$ if $D_b$
is a sound-hard obstacle. Further, we have $\mathscr{B}v:=\pa_\nu v+\beta v$ on an open subset
$\G\subset\pa D_b$ (with the impedance coefficient $\beta$ such that $\I(\beta)\ge0$) and
$\mathscr{B}v:=v$ on $\pa D_b\se\ov{\G}$ if $D_b$ is a mixed-type obstacle.

In the case $\la=1$ the scattering problem (\ref{1.1})-(\ref{1.4}) becomes
\be\label{1.1+}
\Delta u+k^2nu=0 && \mbox{in}\;\;\R^3\se\ov{D_b},\\ \label{1.4a+}
\mathscr{B}(v)=0 &&\mbox{on}\;\;\pa D_b,\\ \label{1.4+}
\lim_{r\rightarrow\infty}r\Big(\frac{\pa u^s}{\pa r}-iku^s\Big)=0 && r=|x|,
\en
which is also called {\em the medium scattering problem with embedded objects}.

The condition (\ref{1.4}) (or (\ref{1.4+})) is referred to the {\em Sommerfeld radiation condition}
which allows the following asymptotic behavior of the scattered field $u^s$:
\be\label{1.5}
u^s(x;d)=\frac{e^{ik|x|}}{4\pi|x|}\left\{u^{\infty}(\wi{x};d)+O\Big(\frac{1}{|x|}\Big)\right\},
\qquad |x|\rightarrow\infty
\en
uniformly in all directions $\wi{x}=x/|x|\in\Sp^2$, where $u^\infty$ is defined on the
unit sphere $\Sp^2$ and known as the far field pattern of the scattered field $u^s$.

We also consider the problem of scattering of a time-harmonic electromagnetic plane wave
by the inhomogeneous penetrable obstacle $D$ with the embedded obstacle $D_b$ and surrounded
by a homogeneous background medium. This problem can be formulated as follows:
\be\label{1.6}
\curl E-ikH=0,\quad \curl H+ikE=0 && \mbox{in}\;\;\R^3\se\ov{D},\\ \label{1.7}
\curl G-ikF=0,\quad \curl F+ikn(x)G=0 && \mbox{in}\;\;D\se\ov{D_b},\\ \label{1.8}
\nu\times E-\nu\times G=0,\quad
\nu\times H-\la_H\nu\times F=0 && \mbox{on}\;\;\pa D,\\ \label{1.9a}
\mathscr{B}_E (G)=0  && \mbox{on}\;\;\pa D_b,\\ \label{1.9}
\lim_{r\rightarrow\infty}r(H^s\ti\wi{x}-E^s)=0 && r=|x|,
\en
where $E,G$ are the electric fields, $H,F$ are the magnetic fields,
$E=E^i+E^s$ and $H=H^i+H^s$ in $\R^3\se\ov{D}$ with the incident plane wave
\ben
E^i(x)=\frac{i}{k}\curl\curl pe^{ikx\cdot d},\qquad
H^i(x)=\frac{1}{ik}\curl E^i(x).
\enn
Here, $\la_H=\mu_0/\mu_1$, $k^2=\om^2\vep_0\mu_0$ is the wave number,
$n=(\vep_1+i\sig_1/\om)\mu_1/(\vep_0\mu_0)$ is the refractive index of the inhomogeneous medium in
$D\se\ov{D}_b$ with electric permittivity $\vep_1,$ magnetic permeability $\mu_1$
and electric conductivity $\sig_1\ge0$ differing from the electric permittivity $\vep_0,$
magnetic permeability $\mu_0$ and electric conductivity $\sig_0=0$ of the surrounding medium $\R^3\se\ov{D}$,
$d$ is the incident direction and $p$ is the polarization vector.
Similarly, $\mathscr{B}_E$ denotes the boundary condition on $\pa D_b$, which corresponds to a
perfect conductor condition if $\mathscr{B}_E (G):=\nu\times G=0$ and an impedance condition
if $\mathscr{B}_E(G):=\nu\times \curl G-i\rho_E(\nu\times G)\times\nu=0$ with a positive constant
$\rho_E$.
In addition, the condition (\ref{1.9}) is known as the {\em Silver-M\"{u}ller radiation condition},
which leas to the asymptotic behaviors:
\be\label{1.0}
E^s(x)&=&\frac{e^{ik|x|}}{|x|}\left\{E^\infty(\wi{x};d;p)+O\Big(\frac{1}{|x|}\Big)\right\},
\qquad |x|\rightarrow\infty\\
H^s(x)&=&\frac{e^{ik|x|}}{|x|}\left\{H^\infty(\wi{x};d;p)+O\Big(\frac{1}{|x|}\Big)\right\},
\qquad |x|\rightarrow\infty
\en
uniformly for all $\wi{x}=x/|x|\in\Sp^2$, where $E^\infty$ and $H^\infty(=\wi{x}\times E^\infty)$
defined on $\Sp^2$ are called the far field patterns of the electric field $E^s$ and the magnetic
field $H^s$, respectively.

In the case $\la_H=1$ we consider {\em the medium scattering problem without embedded obstacles}:
\be\label{1.6+}
\curl E-ikH=0,\quad \curl H+ikn(x)E=0 && \mbox{in}\;\;\R^3,\\ \label{1.9+}
\lim_{r\rightarrow\infty}r(H^s\ti\wi{x}-E^s)=0 && r=|x|,
\en
where $E=E^i+E^s$ and $H=H^i+H^s$ in $\R^3$, $n=1$ in $\R^3\se\ov{D}$ and
$n=(\vep_1+i\sig_1/\om)\mu_1/(\vep_0\mu_0)$ in $D$ with $n\in L^\infty(D)$.

The existence of a unique solution to the transmission scattering problems (\ref{1.1})-(\ref{1.4})
and (\ref{1.6})-(\ref{1.9}) (or the medium scattering problems (\ref{1.1+})-(\ref{1.4+}) and
(\ref{1.6+})-(\ref{1.9+})) can be established by the variational approach
or the integral equation method \cite{LZ,LZY,LZH,LZaa,LZ12}.
In this paper, we will assume that the transmission scattering problems (\ref{1.1})-(\ref{1.4})
and (\ref{1.6})-(\ref{1.9}) (or the medium scattering problems (\ref{1.1+})-(\ref{1.4+}) and
(\ref{1.6+})-(\ref{1.9+})) are well-posed and study the {\em inverse scattering problem:}
given $k$, $\la$ or $\la_H$, determine the penetrable obstacle $D$ (or the support $D$ of the inhomogeneous
medium in the case $\la=1$ or $\la_H=1$) from a knowledge of $u^\infty(\wi{x};d)$ or $E^\infty(\wi{x};d;p)$
for all $\wi{x},d\in\Sp^2$ and $p\in\R^3$, disregarding its contents $n$ and $D_b$.

In the case $D_b=\emptyset$, many uniqueness results have been obtained in determining the penetrable obstacle
$D$. The first such uniqueness result was established by Isakov \cite{IV} in 1990.
The idea is to construct singular solutions of the scattering problem with respect to
two different penetrable obstacles with identical far-field patterns, based on the variational method.
In 1993, Kirsch and Kress \cite{KK} greatly simplified Isakov's method by considering classical scattering 
solutions and using the integral equation technique to establish a priori estimates of the solution on some 
part of the interface $\pa D$. In \cite{KK} the method was also extended to the case of Neumann boundary 
conditions (corresponding to impenetrable, sound-hard obstacles).
Since then, the idea has been extensively studied and applied to establish uniqueness results for
many other inverse scattering problems with transmission or conductive boundary conditions as well as other 
boundary conditions (see, e.g., \cite{GK96,HF,IV08,KP,LZH,LZaa,LZ12,MM,V04,YZ} and the references quoted there).
The idea of Isakov has also been modified to establish uniqueness results for inverse electromagnetic
scattering problems both by a penetrable, inhomogeneous, isotropic obstacle $D$ in \cite{HP0}
under the condition that the boundary $\pa D$ is in $C^{2,\al}$ with $0<\al<1,$
the refractive index $n\in C^{1,\al}(\ov{D})$ is a constant near the boundary $\pa D$
and $\I(n(x_0))>0$ for some $x_0\in D$ and by a penetrable, homogeneous, isotropic obstacle coated
with a thin conductive layer in \cite{HF1}.
In \cite{HP1}, H\"ahner introduced a different technique to prove the unique determination of a penetrable, 
inhomogeneous, anisotropic obstacle $D$ from a knowledge of the scattered near-fields for all incident 
plane waves. The method of H\"ahner is based on a study of the existence, uniqueness and regularity of
solutions to the corresponding interior transmission problem in $D$.
In \cite{CC0} Cakoni and Colton extended H\"ahner's idea to deal with the case with a penetrable,
inhomogeneous, anisotropic obstacle possibly partly coated with a thin layer of a highly
conductive material. It seems difficult to apply the idea in \cite{CC0,HP1} to the multi-layered case
and the case with embedded obstacles.
Further, it was recently proved in \cite{EH15,HSV16} that a penetrable, convex polyhedron or polygon obstacle
can be uniquely determined by the far-field pattern over all observation directions incited by a single
incident plane wave. The arguments used in \cite{EH15,HSV16} rely essentially on the expansion of solutions
to the Helmholtz equation. 
Furthermore, it was proved in \cite{NUW11} that a penetrable obstacle with a $C^2$-smooth boundary
in a two-dimensional domain can be uniquely reconstructed from acoustic measurements made on the boundary
of the domain. The method used in \cite{NUW11} uses complex geometrical optics solutions to the Helmholtz
equation with polynomial-type phase functions. The result in \cite{NUW11} was extended to the three-dimensional
case in \cite{Y10}, to the case with Lipschitz continuous interfaces in \cite{SY12} and to the isotropic
Maxwell system in \cite{KS14a}. These results were further generalized to the elastic wave case with the
far-field measurements in \cite{KS14}, based on considering complex geometrical optics solutions to the
Lame system with linear or logarithmic phase functions and using $L^p$ estimates of the gradients of
the solutions to the Lame systems with discontinuous Lame coefficients, and to the anisotropic Maxwell
system in \cite{KLS15} by constructing oscillating-decaying-type solutions to the anisotropic Maxwell system.

In the case when there are embedded obstacles in the penetrable obstacle or in an inhomogeneous medium, 
that is, $D_b\not=\emptyset$, it was proved in \cite{LZ} that the penetrable obstacle $D$ and the embedded 
obstacle $D_b$ can be simultaneously determined from knowledge of the acoustic far-field pattern for incident 
plane waves under the condition that $n$ is a known constant in $D$. 
By employing the technique proposed in \cite{HP0} the uniqueness result was established in \cite{LZY} for
determining the penetrable obstacle $D$ and the embedded obstacle $D_b$ simultaneously from knowledge of 
the electric far-field pattern for incident plane waves provided $n$ is a known complex constant with
positive imaginary part in $D$. 
In \cite{EH11}, Elschner and Hu considered the inverse transmission scattering problem by a two-dimensional, 
impenetrable obstacle surrounded by an unknown piecewise homogeneous medium and proved that the far-field 
patterns for all incident and observation directions at a fixed frequency uniquely determine the unknown 
surrounding medium as well as the impenetrable obstacle. Their method is based on constructing the Green 
function to a two-dimensional elliptic equation with piecewise constant leading coefficients associated
with the direct scattering problem and studying the singularity of the Green function when
the point source position approaches the interfaces and the impenetrable obstacle.
In \cite{LZZ15}, the uniqueness result was proved in determining the scattering support of a complex 
scatterer, possibly consisting of an inhomogeneous medium and impenetrable obstacles, by the acoustic 
far-field measurements. The technique used in \cite{LZZ15} is based essentially on Isakov's idea 
in conjunction with the integral equation method and the singular point source with second-order 
singularity. However, it is difficult to extend the technique of \cite{LZZ15} 
to the case of Maxwell¡¯s equations.

It should be pointed out that all the above uniqueness results were obtained under the assumption
that the transmission coefficient $\la\not=1$ or $\la_H\not=1$ for the isotropic case
or the matrix characterizing the anisotropic medium is different from the identity matrix $I$.
In this paper, we propose a new technique to establish uniqueness results for our inverse scattering 
problem, that is, uniqueness results in determining the penetrable obstacle $D$ 
(or the support $D$ of the inhomogeneous medium in the case $\la=1$ or $\la_H=1$) from knowledge of 
the acoustic far-field measurements or the electric far-field measurements at a fixed frequency, 
disregarding its contents $n$ and $D_b$.
Our method is based on constructing a well-posed interior transmission problem
in a {\em small} domain inside $D$ associated with the Helmholtz or Maxwell equations.
Here, a key role is played by the smallness of the domain which ensures, for the case
$\la=1$ or $\la_H=1$, that the lowest transmission eigenvalue is large so that
a given wave number $k$ is not an eigenvalue of the constructed interior transmission problem.
This is different from the method used in \cite{CC0,HP1}, where the interior transmission problem 
considered is defined in the {\em whole penetrable obstacle} $D$ and may have interior transmission 
eigenvalues, so the case $\la=1$ or $\la_H=1$ is excluded.
Another ingredient in our proofs is a priori estimates of solutions to the transmission scattering problems
with data in $L^p$ ($1<p<2$) which will be established in this paper by using the integral equation method.
These a priori estimates are also expected to be useful on their own right.
Our method works for the cases either $\la=1$ and $\la_H=1$ or $\la\not=1$ and $\la_H\not=1$ and 
can deal with the acoustic and electromagnetic cases in a unified way.
Moreover, our method can also deal with the case with unbounded interfaces, as seen in \cite{LZ16}.
It should be remarked that reconstruction algorithms, based on the factorization method \cite{KG08},
have been developed in \cite{QYZ17,YZZ} to reconstruct the penetrable obstacle numerically, 
disregarding its contents.

It is well known that the existence and distribution of the eigenvalues of interior transmission
problems play an important role in the linear sampling method \cite{CC1} and
the factorization method \cite{KG08}.
Thus, the existence and computation of the eigenvalues of interior transmission problems have
been extensively studied recently (see, e.g. \cite{CC1,CGH,CH,CMS,SJ0} and the references there).
In particular, it was proved in \cite{CGH} that the lowest transmission eigenvalue trends to
infinity as the radius of the domain in which the interior transmission problem is defined
trends to zero. Thus, for a given wave number $k$ the domain can be taken to be small enough
so that $k$ is not an eigenvalue of the interior transmission problem.
Our method is motivated by this observation.

The remaining part of the paper is organized as follows. In Sections \ref{sec2} and \ref{sec3}
we consider the inverse acoustic and electromagnetic scattering problems by penetrable obstacles,
respectively. We also utilize the integral equation method to establish
a priori estimates of solutions to the acoustic and electromagnetic transmission problems
with data in $L^p$ ($1<p<2$), which are used in our uniqueness proofs of the
inverse scattering problems. It is expected that these a priori estimates are also useful
in other applications.

\section{The inverse acoustic scattering problem} \label{sec2}
\setcounter{equation}{0}

In this section we introduce the new technique to prove the unique determination of the inhomogeneous
penetrable obstacle $D$ (or the support $D$ of the inhomogeneous medium in the case $\la=1$) from the
far-field pattern $u^\infty(\wi{x};d)$ for all $\wi{x},d\in S^2$, disregarding its contents $n$ and $D_b$.
Our method is based on constructing a well-posed interior transmission problem in a {\em small} domain
associated with the Helmholtz or modified Helmholtz equation. Here, a key role is played by the smallness
of the domain which ensures that the given wave number $k$ is not a transmission eigenvalue of
the constructed interior transmission problem for the case $\la=1$.
It should be noted that all the previous methods do not work for the case $\la=1$.
For the case $\la\not=1$ which has been considered previously,
our method gives a simplified proof. Furthermore, our method also works for the electromagnetic
case, as shown in the next section, and for the case of unbounded interfaces (see \cite{LZ16}).

\subsection{Interior transmission problems}

Let $\Om\subset\R^3$ be a simply connected and bounded domain with $\pa\Om\in C^2$.
In the case $\la\neq1$ we consider the following modified interior transmission problem (MITP):
\be\label{2.8}
\tr U-U=\rho_1 &&\text{in}\;\;\Om, \\ \label{2.9}
\tr V-V=\rho_2 &&\text{in}\;\;\Om, \\ \label{2.10}
U-V=f_1,\;\la\frac{\pa U}{\pa\nu}-\frac{\pa V}{\pa\nu}=f_2 && \text{on}\;\;\pa\Om,
\en
where $\rho_0,\rho_1\in L^2(\Om)$, $f_1\in H^{{1}/{2}}(\pa\Om)$ and $f_2\in H^{-{1}/{2}}(\pa\Om)$.
This problem has been studied in \cite{CC1}, and the following result
was obtained (see \cite[Theorem 6.7]{CC1}).

\begin{lemma}\label{le2.2}\rm{{(\cite[Theorem 6.7]{CC1})}}
If $\la\not=1$ then the problem (MITP) has a unique solution $(U,V)\in H^1(\Om)\times H^1(\Om)$
such that
\ben
\|U\|_{H^1(\Om)}+\|V\|_{H^1(\Om)}\le C\big(\|\rho_1\|_{L^2(\Om)}+\|\rho_2\|_{L^2(\Om)}
+\|f_1\|_{H^{{1}/{2}}(\pa\Om)}+\|f_2\|_{H^{-{1}/{2}}(\pa\Om)}\big).
\enn
\end{lemma}

In the case $\la=1$ we consider the following interior transmission problem (ITP):
\be\label{2.1}
\tr U+k^2n(x)U=0 && \text{in}\;\;\Om, \\ \label{2.2}
\tr V+k^2V=0 && \text{in}\;\;\Om, \\ \label{2.3}
    U-V=f_1,\;\;\frac{\pa U}{\pa\nu}-\frac{\pa V}{\pa\nu}=f_2 && \text{on}\;\;\pa\Om,
\en
where $f_1\in H^{{1}/{2}}(\pa\Om)$ and $f_2\in H^{-{1}/{2}}(\pa\Om)$.
This problem has been studied in \cite{CGH}.

Let $w:=U-V$. Then it is easy to see that $w$ satisfies the fourth-order equation
\be\label{2.4}
(\tr+k^2n)\frac{1}{n-1}(\tr+k^2)w=0
\en
with the boundary conditions $\g_0w=f_1$ and $\g_1w=f_2$. Here, $\g_j$ $(j=0,1)$ denotes
the $j$th-order trace operator.

Define the Hilbert space
\ben
H^1_\tr(\Om)=\{w\in H^1(\Om):\;\Delta w\in L^2(\Om)\}
\enn
with the norm $\|w\|^2_{H^1_\Delta(\Om)}=\|w\|^2_{H^1(\Om)}+\|\Delta w\|^2_{L^2(\Om)}$.
Using the Green's theorem, we easily prove that $\g_0w\in H^{{1}/{2}}(\pa\Om)$,
$\g_1w\in H^{-{1}/{2}}(\pa\Om)$. In particular, if $\g_0w=\g_1w=0$ for
all $w\in H^1_\Delta(\Om)$, then $H^1_\Delta(\Om)=H^2_0(\Om)$.

We assume that the data $f_1\in H^{{1}/{2}}(\pa\Om)$ and $f_2\in H^{-{1}/{2}}(\pa\Om)$
satisfy the condition (\textbf{C}) with some $w_0\in H^1_\Delta(\Om),$ that is,
there exists a function $w_0\in H^1_\Delta(\Om)$ such that $\g_0w_0=f_1$, $\g_1w_0=f_2.$
Then the interior transmission problem (ITP) is equivalent to the variational problem:
Find $w\in H^1_\Delta(\Om)$ with $\g_0w=f_1$ and $\g_1w=f_2$ such that
\be\label{2.5}
a(w,h):=\int_{\Om}\frac{1}{n-1}(\tr+k^2)w(\Delta+k^2n)\ov{h}dx=0\quad
\text{for\;all\;}\;h\in H^2_0(\Om).
\en
Let $\wid{w}:=w-w_0\in H^2_0(\Om)$. Then the variational problem (\ref{2.5}) is equivalent to
the problem: Find $\wid{w}\in H^2_0(\Om)$ such that
\be\label{2.6}
a(\wid{w},h)=-a(w_0,h)\quad\text{for\;all\;}\;h\in H^2_0(\Om).
\en
Based on (\ref{2.6}), the following result can be established (see \cite{CGH} for a proof).

\begin{lemma}\label{le2.1}\rm{(\cite[Lemma 2.4]{CGH})}
If $n(x)>1+r_0$ or $0<n(x)<1-r_1$ with some constants $r_0,r_1>0$, then
\ben
a(\wid{w},\wid{w})\geq C\|\wid{w}\|^2_{H^2_0(\Om)},\quad\forall\wid{w}\in H^2_0(\Om)
\enn
for $0<k^2<\min\{\la_1(\Om),\la_1(\Om)/\sup(n)\},$ where $\la_1(\Om)$ is the first Dirichlet
eigenvalue of the operator $-\tr$ in $\Om$.
\end{lemma}

By Lemma \ref{le2.1} the following result can be easily obtained.

\begin{corollary}\label{co2.1}
Assume that $f_1,f_2$ satisfy the condition (\textbf{C}) with $w_0\in H^1_\Delta(\Om).$
For any fixed $k>0$, if the diameter of $\Om$ is small enough (so $\la_1(\Om)$ is
large enough) so that $k^2<\min\{\la_1(\Om),\la_1(\Om)/\sup(n)\}$,
then the interior transmission problem (ITP)
has a unique solution $(U,V)\in L^2(\Om)\times L^2(\Om)$ with
\be\label{2.7}
\|U\|_{L^2(\Om)}+\|V\|_{L^2(\Om)}\leq C\|w_0\|_{H^1_\Delta(\Om)}.
\en
\end{corollary}

\begin{proof}
For any fixed $k>0$, if the diameter of $\Om$ is small enough
so that $k^2<\min\{\la_1(\Om),\la_1(\Om)/\sup(n)\}$, then, by Lemma \ref{le2.1} it follows that
\ben
a(\wid{w},\wid{w})\geq C\|\wid{w}\|^2_{H^2_0(\Om)}\quad\mbox{for all}\;\;\wid{w}\in H^2_0(\Om).
\enn
This, together with the Lax-Milgram theorem, implies that the variational problem (\ref{2.6})
has a unique solution $\wid{w}\in H^2_0(\Om)$ satisfying the estimate
\be\label{2.7a}
\|\wid{w}\|_{H^2_0(\Om)}\le C\|w_0\|_{H^1_\Delta(\Om)}.
\en
Define $U:=[{1}/({n-1})](\Delta+k^2)w$, $V:=U-w$. Then it is easy to see that
$(U,V)\in L^2(\Om)\times L^2(\Om)$, with $U-V\in H^2_0(\Om)$, is the unique solution to the
interior transmission problem (ITP). The estimate (\ref{2.7}) follows easily from (\ref{2.7a})
and the fact that $w=\wid{w}+w_0$.
\end{proof}

\begin{remark}\label{rm2.1}{\rm
The uniqueness result for the case $\la\not=1$ corresponds to the well-posed problem (MITP),
whilst that for the case $\la=1$ corresponds to the much harder problem (ITP) which is not necessarily
well-posed for all wavenumbers $k$ if $\Om$ is not small, as shown in Lemma \ref{le2.2}
and Corollary \ref{co2.1}. This explains clearly why the transmission coefficient $\la$ is
assumed not to be $1$ (i.e., $\la\not=1$) in all the previous methods of the uniqueness proofs
of the inverse problems.
}
\end{remark}

\subsection{A priori estimates for the transmission problems with $L^p$ boundary data}

In this subsection, we establish a priori estimates of solutions to the acoustic transmission
problem with boundary data in $L^p$ ($1<p\le 2$), employing the integral equation method.
These a priori estimates are needed later in the uniqueness proof of the inverse problem
and are also interesting on their own right.

Consider the general acoustic transmission problem
\be\label{2.11}
\tr w_1+k^2w_1=0 && \text{in}\;\;\R^3\se\ov{D},\\ \label{2.12}
\tr w_2+k^2n(x)w_2=0 &&\text{in}\;\;D\se\ov{D_b},\\ \label{2.13}
w_1-\g w_2=f_1,\;\;
\frac{\pa w_1}{\pa\nu}-\frac{\pa w_2}{\pa\nu}=f_2 &&\text{on}\;\;\pa D,\\ \label{2.14a}
\mathscr{B}(w_2)=0 && \text{on}\;\;\pa D_b, \\ \label{2.14}
\frac{\pa w_1}{\pa r}-ikw_1=o\Big(\frac{1}{r}\Big) && r=|x|\rightarrow\ify,
\en
where $f_1,f_2\in L^p(\pa D)$ with $1<p<2$ and $\g=1/\la$.

We introduce the single- and double-layer boundary operators
\ben
(S_{\rm{e}\rm{e}}\phi)(x)&:=&\int_{\pa D}\Phi(x,y)\phi(y)ds(y),\quad x\in\pa D,\\
(K_{\rm{e}\rm{e}}\phi)(x)&:=&\int_{\pa D}\frac{\pa\Phi(x,y)}{\pa\nu(y)}\phi(y)ds(y),\quad x\in\pa D
\enn
and the their normal derivative operators
\ben
(K'_{\rm{e}\rm{e}}\phi)(x)&:=&\int_{\pa D}\frac{\pa\Phi(x,y)}{\pa\nu(x)}\phi(y)ds(y),\quad x\in\pa D,\\
(T_{\rm{e}\rm{e}}\phi)(x)&:=&\frac{\pa}{\pa\nu(x)}\int_{\pa D}\frac{\pa\Phi(x,y)}{\pa\nu(y)}\phi(y)ds(y),
  \quad x\in\pa D.
\enn
Similarly, we also introduce the boundary operators $S_{\rm{i}\rm{i}}$, $K_{\rm{i}\rm{i}}$,
$K'_{\rm{i}\rm{i}}$ and $T_{\rm{i}\rm{i}}$ defined on $\pa D_b$ as well as
$S_{\rm{t}\rm{h}}$, $K_{\rm{t}\rm{h}}$, $K'_{\rm{t}\rm{h}}$ and $T_{\rm{t}\rm{h}}$ with
$\rm{t,h}=\rm{e,i}$, respectively, where, for example, $S_{\rm{e}\rm{i}}$ is defined similarly
as $S_{\rm{e}\rm{e}}$ but with $x\in \pa D_b$.
It follows from \cite[Lemma 9]{P0} and \cite[Lemma 1]{P1} that the operators
$S_{\rm{p}\rm{p}},K_{\rm{p}\rm{p}}$ and $K'_{\rm{p}\rm{p}}$ with $\rm{p}=\rm{e,i}$ are both
bounded and compact in $L^q(\pa D)$ ($1<q<\infty$).

\begin{theorem}\label{thm2.3}
For $f_1,f_2\in L^p(\pa D)$ with $4/3\le p\le2$ the transmission problem $(\ref{2.11})-(\ref{2.14})$
has a unique solution $(w_1,w_2)\in L^2_{\rm{loc}}(\R^3\se\ov{D})\times L^2(D_b)$ satisfying that
\be\label{2.15}
\|w_1\|_{L^2_{\rm{loc}}(\R^3\se\ov{D})}+\|w_2\|_{L^2(D\se\ov{D_b})}
\leq C(\|f_1\|_{L^p(\pa D)}+\|f_2\|_{L^p(\pa D)}).
\en
\end{theorem}

\begin{proof}
We only consider the case with an impedance condition on $\pa D_b$, that is,
$\mathscr{B}(w_2)=\pa w_2/\pa\nu+i\rho w_2=0$. The other case can be treated similarly.

Step 1. Assume that $k^2n(x)\equiv k_1^2>0$ is a constant.
We seek a solution of the problem (\ref{2.11})-(\ref{2.14}) in the form
\be\label{2.16}
w_1(x)&=&\int_{\pa D}\Phi(x,y)\phi(y)ds(y)
    +\int_{\pa D}\frac{\pa\Phi(x,y)}{\pa\nu(y)}\psi(y)ds(y),\quad x\in\R^3\se\ov{D}\\ \no
w_2(x)&=&\int_{\pa D}\Phi_1(x,y)\phi(y)ds(y)
                   +\int_{\pa D}\frac{\pa\Phi_1(x,y)}{\pa\nu(y)}\psi(y)ds(y)\\ \label{2.17}
       &&+\int_{\pa D_b}\Phi_1(x,y)\eta(y)ds(y), \qquad x\in D\se\ov{D_b},
\en
where $\Phi(x,y)=\exp({ik|x-y|})/({4\pi|x-y|})$ and $\Phi_1(x,y)=\exp({ik_1|x-y|})/({4\pi|x-y|})$.

Then, by the jump relations of the layer potentials (see \cite{P1} for the case in $L^p$
and \cite{CK,CK1} for the case in spaces of continuous functions),
the transmission problem (\ref{2.11})-(\ref{2.14}) can be reduced to the system of integral equations
\be\label{2.18}
\left(\begin{array}{c}
    \psi \\
    \phi \\
    \eta\\
  \end{array}\right)
+L\left(\begin{array}{c}
    \psi \\
    \phi \\
    \eta\\
  \end{array}\right)
=\left(\begin{array}{c}
    hf_1 \\
    -f_2 \\
    0\\
  \end{array}\right)
\qquad \text{in}\;\;L^p(\pa D)\times L^p(\pa D)\times C(\pa D_b),
\en
where $h:=2/(1+\g)$ and the operator $L$ is given by
\ben
L:=\left(\begin{array}{ccc}
    h(K_{\rm{e}\rm{e}}-\g K^{(1)}_{\rm{e}\rm{e}}) & h(S_{\rm{e}\rm{e}}-\g S^{(1)}_{\rm{e}\rm{e}})
    & -h\g S^{(1)}_{\rm{i}\rm{e}}\\
    T^{(1)}_{\rm{e}\rm{e}}-T_{\rm{e}\rm{e}} & K'^{(1)}_{\rm{e}\rm{e}}-K'_{\rm{e}\rm{e}}
    & K'^{(1)}_{\rm{i}\rm{e}} \\
    -2(T^{(1)}_{\rm{e}\rm{i}}+i\rho K^{(1)}_{\rm{e}\rm{i}})
    &-2 (K'^{(1)}_{\rm{e}\rm{i}}+i\rho S^{(1)}_{\rm{e}\rm{i}})
     & -2(K'^{(1)}_{\rm{i}\rm{i}}+i\rho S^{(1)}_{\rm{i}\rm{i}})
  \end{array}\right).
\enn
Here, the operators $S^{(1)}_{\rm{t}\rm{h}},K^{(1)}_{\rm{t}\rm{h}},K'^{(1)}_{\rm{t}\rm{h}}$
and $T^{(1)}_{\rm{t}\rm{h}}$ with $\rm{t,h}=\rm{e,i}$, respectively, are defined similarly as
$S_{\rm{t}\rm{h}},K_{\rm{t}\rm{h}},K'_{\rm{t}\rm{h}}$ and $T_{\rm{t}\rm{h}}$ with the kernel
$\Phi(x,y)$ replaced by $\Phi_1(x,y)$.
It is easy to see that (\ref{2.18}) is of Fredholm type since
the elements of $L$ are all compact operators in the corresponding Banach spaces.
This, together with the uniqueness of the scattering problem (\ref{1.1})-(\ref{1.4}), implies that
(\ref{2.18}) has a unique solution $(\psi,\phi,\eta)^T\in L^p(\pa D)\times L^p(\pa D)\times C(\pa D_b)$
satisfying the estimate
\be\label{2.19}
\|\psi\|_{L^p(\pa D)}+\|\phi\|_{L^p(\pa D)} +\|\eta\|_{L^\ify(\pa D_b)}
\leq C(\|f_1\|_{L^p(\pa D)}+\|f_2\|_{L^p(\pa D)}).
\en
Therefore, we obtain that
\be\no
&&\left\|\int_{\pa D}\Phi_1(\cdot,y)\phi(y)ds(y)\right\|_{L^2(D)}
=\sup_{g\in L^2,\|g\|_{L^2(D)}=1}
\left|\int_{D}\left\{\int_{\pa D}\Phi_1(x,y)\phi(y)ds(y)\right\}g(x)dx\right|\\ \no
&&\qquad\qquad=\sup_{g\in L^2,\|g\|_{L^2(D)}=1}
\left|\int_{\pa D}\left\{\int_{D}\Phi_1(x,y)g(x)dx\right\}\phi(y)ds(y)\right|\\ \no
&& \qquad\qquad\leq |\pa D|^{{1}/{q}}\sup_{g\in L^2,\|g\|_{L^2(D)}=1}
\sup_{y\in\pa D}\|\Phi_1(\cdot,y)\|_{L^2(D)}\|g\|_{L^2(D)}\|\phi\|_{L^p(\pa D)}\\ \label{2.20}
&&\qquad\qquad=|\pa D|^{{1}/{q}}\sup_{y\in\pa D}\|\Phi_1(\cdot,y)\|_{L^2(D)}\|\phi\|_{L^p(\pa D)}
\en
and
\be\no
&&\left\|\int_{\pa D}\frac{\pa \Phi_1(\cdot,y)}{\pa\nu(y)}\psi(y)ds(y)\right\|_{L^2(D)}
=\sup_{g\in L^2,\|g\|_{L^2(D)}=1}
\left|\int_{D}\left\{\int_{\pa D}\frac{\pa\Phi_1(x,y)}{\pa\nu(y)}\psi(y)ds(y)\right\}g(x)dx\right|\\ \no
&&\qquad\qquad=\sup_{g\in L^2,\|g\|_{L^2(D)}=1}
\left|\int_{\pa D}\left\{\frac{\pa}{\pa\nu(y)}\int_{D}\Phi_1(x,y)g(x)dx\right\}\psi(y)ds(y)\right|\\ \no
&&\qquad\qquad \leq \sup_{g\in L^2,\|g\|_{L^2(D)}=1}
\left\|\frac{\pa}{\pa\nu(y)}\int_{D}\Phi_1(x,\cdot)g(x)dx\right\|_{L^q(\pa D)}
\cdot\left\|\psi\right\|_{L^p(\pa D)}\\ \label{2.21}
&&\qquad\qquad \leq \sup_{g\in L^2,\|g\|_{L^2(D)}=1}
C\|g\|_{L^2(D)}\cdot\left\|\psi\right\|_{L^p(\pa D)}
=C\left\|\psi\right\|_{L^p(\pa D)}
\en
with $1/p+1/q=1$. Here, we have used the fact that the volume potential operator is
bounded from $L^2(D)$ into $W^{2,2}(D)$ (see \cite[Theorem 9.9]{DT}),
and the boundary trace operator is bounded from $W^{1,2}(D)$ into $L^q(\pa D)$ for $2\le q\le 4$
(see \cite[Theorem 5.36]{AF}).
Further, we derive from \cite{CK} that
\be\label{AKa}
\left\|\int_{\pa D_b}\Phi_1(\cdot,y)\eta(y)ds \right\|_{L^\ify(\R^3)} & \leq & C\|\eta\|_{L^\ify(\pa D_b)}.
\en
Then the desired estimate (\ref{2.15}) follows from (\ref{2.16})-(\ref{2.17}) and
(\ref{2.19})-(\ref{AKa}) in the case when $k^2n(x)\equiv k^2_1$.

Step 2. For the general case $n\in L^\infty(D\se\ov{D}_b)$, we consider the following problem
\be\label{2.22}
\tr W_1+k^2W_1=0&& \text{in}\;\;\R^3\se\ov{D},\\ \label{2.23}
\tr W_2+k^2n(x)W_2=g&& \text{in}\;\;D\se\ov{D_b},\\ \label{2.24}
W_1-\g W_2=0,\;\;\frac{\pa W_1}{\pa\nu}-\frac{\pa W_2}{\pa\nu}=0&& \text{on}\;\;\pa D,\\ \label{2.25a}
\mathscr{B}(W_2)=0 && \text{on}\;\;\pa D_b,\\ \label{2.25}
\frac{\pa W_1}{\pa r}-ikW_1=o\Big(\frac{1}{r}\Big)&& r=|x|\rightarrow\infty,
\en
where $g:=(k_1^2-k^2n(x))\wid{w}_2\in L^2(D\se\ov{D}_b)$ and $(\wid{w}_1,\wid{w}_2)$ denotes the solution
of the problem (\ref{2.11})-(\ref{2.14}) with $k^2n(x)\equiv k^2_1$.
By Step 1, we have
\be\label{2.15+}
\|\wid{w}_1\|_{L^2_{\rm{loc}}(\R^3\se\ov{D})}+\|\wid{w}_2\|_{L^2(D\se\ov{D_b})}
\le C(\|f_1\|_{L^p(\pa D)}+\|f_2\|_{L^p(\pa D)}).
\en
By using the variational method, it can be easily proved that for every $g\in L^2(D\se\ov{D_b})$
the problem (\ref{2.22})-(\ref{2.25}) has a unique solution
$(W_1,W_2)\in H^1_{\rm{loc}}(\R^3\se\ov{D})\times H^1(D\se\ov{D_b})$ satisfying the estimate
\be\label{2.26}
\|W_1\|_{H^1_{\rm{loc}}(\R^3\se\ov{D})}+\|W_2\|_{H^1(D\se\ov{D_b})}\leq C\|g\|_{L^2(D\se\ov{D_b})}.
\en
Define $w_1:=W_1+\wid{w}_1$ and $w_2:=W_2+\wid{w}_2$.
Then from (\ref{2.15+}) and (\ref{2.26}) it follows that
$(w_1,w_2)\in L^2_{\rm{loc}}(\R^3\se\ov{D})\times L^2(D\se\ov{D_b})$ is the unique solution of
the problem (\ref{2.11})-(\ref{2.14}) satisfying the estimate (\ref{2.15}).
The proof is thus complete.
\end{proof}

\begin{corollary}\label{co2.4}
For $z^*\in\pa D$ and for a sufficiently small $\delta>0$ define
$z_j:=z^*+({\delta}/{j})\nu(z^*)\in\R^3\se\ov{D}$, $j\in\N$.
Let $(u_j,v_j)$ be the solution of the transmission problem $(\ref{1.1})-(\ref{1.4})$ corresponding to
the incident point source $u^i_j=\Phi(x,z_j)$, $j\in\N$. Then
\be\label{2.27}
\|v_j\|_{L^2(D\se\ov{D_b})}\le C
\en
uniformly for $j\in\N$.
\end{corollary}

\begin{proof}
Let $w_{1j}:=u_j^s-\Phi(x,y_j)$ and $w_{2j}:=\la v_j$ with $y_j:=z^*-({\delta}/{j})\nu(z^*)\in D\se\ov{D_b}$.
Then $(w_1,w_2)$ satisfies the problem (\ref{2.11})-(\ref{2.14}) with
\ben
f_1=f_{1j}&:=&-\Phi(z,z_j)-\Phi(x,y_j),\\
f_2=f_{2j}&:=&-\frac{\pa\Phi(z,z_j)}{\pa\nu(z)}-\frac{\pa\Phi(z,y_j)}{\pa\nu(z)}.
\enn
Obviously, $f_{1j}\in L^p(\pa D)$ is uniformly bounded for $j\in\N$, where $1<p<2$.
Further, from \cite[Lemma 4.2]{CKM} it is seen that $f_{2j}\in C(\pa D)$  is uniformly bounded
for $j\in\N$, so $f_{2j}\in L^p(\pa D)$ is uniformly bounded for $j\in\N$, where $1<p<2$.
The estimate (\ref{2.27}) then follows from Theorem \ref{thm2.3}.
\end{proof}

\begin{theorem}\label{thm2.5}
Let $(u,v)$ be the solution of the scattering problem $(\ref{1.1+})-(\ref{1.4+})$
corresponding to the supper singular point source $u^i=\nabla_x\Phi(x,z)\cdot\vec{a}$,
where $z\in\R^3\se\ov{D}$ and $\vec{a}\in \R^3$ is a fixed vector.
Then $v\in L^p(D\se\ov{D_b})$ and $v-u^i\in H^1(D\se\ov{D_b})$ such that
\be\label{2.28}
\|v\|_{L^p(D\se\ov{D_b})}+\|v-u^i\|_{H^1(D\se\ov{D_b})}
\leq C(\|u^i\|_{L^p(D\se\ov{D_b})}+\|u^i\|_{L^\ify(\pa D_b)})
\en
for every $p$ with $6/5\le p<2$.
\end{theorem}

To prove the estimate (\ref{2.28}), we reformulate the scattering problem (\ref{1.1+})-(\ref{1.4+})
as an equivalent Lippmann-Schwinger-type equation. To this end, we introduce the exterior boundary
value problem associated with $D_b$ and the boundary condition $\mathscr{B}$:
\be\label{2.28a}
\left\{
 \begin{array}{ll}
 \tr w+k^2w =0 &\text{in}\;\;\R^3\se\ov{D_b}, \\
 \mathscr{B}(w)=f &\text{on}\;\; \pa D_b,\\
 \ds\frac{\pa w}{\pa |x|}-ik w =o\left(\frac{1}{|x|}\right) &  \text{for}\;\;|x|\to\ify.
  \end{array}
\right.
\en
It is well known that the problem (\ref{2.28a}) is well-posed for every $f$ belonging to a suitable
Holder or Sobolev space.

Let $G^s(\cdot,y)$ be the solution to the problem (\ref{2.28a}) with $f=-\mathscr{B}(\Phi(\cdot,y))$
and let $U(\cdot;y)$ be the solution to the problem (\ref{2.28a}) with
$f=-\mathscr{B}(\nabla\Phi(\cdot,y)\cdot\vec{a})$ with $y\in\R^3\se\ov{D_b}$.
Define
\ben
&&G(x,y):=\Phi(x,y)+G^s(x,y) \qquad\qquad\text{in}\;\;\R^3\se\ov{D_b}, \\
&&U^i(x;y):=\nabla_x\Phi(x,y)\cdot\vec{a}+U(x;y)\qquad\text{in}\;\;\R^3\se\ov{D_b}.
\enn
Then $G(x,y)$ satisfies the problem (\ref{2.28a}) with $f=0$, that is, $G(x,y)$ is the Green function of
the problem (\ref{2.28a}) with $f=0$. By the representation theorem for $v(\cdot;z)$, we have
\be\no
v(x;z)&=&\left(\int_{\pa D}-\int_{\pa D_b}\right)
\left\{G(x,y)\frac{\pa v(y;z)}{\pa\nu(y)}-\frac{\pa G(x,y)}{\pa\nu(y)}v(y;z)\right\}ds(y)\\ \label{2.28b}
&&-\int_{D\se\ov{D_b}}k^2[1-n(y)]G(x,y)v(y;z)dy,\qquad\qquad \text{for}\;\;x\in D\se\ov{D_b}.
\en
The integral on $\pa D_b$ obviously equals to zero due to the boundary condition.
Further, the radiation condition for $G(\cdot,x)$ and $v^s(\cdot;z)$ yields
\ben
\int_{\pa D}\left\{G(x,y)\frac{\pa [v(y;z)-U^i(y;z)]}{\pa\nu(y)}
                      -\frac{\pa G(x,y)}{\pa\nu(y)}[v(y;z)-U^i(y;z)]\right\}ds(y) =0.
\enn
This, combined with (\ref{2.28b}) and the Green theorem, implies that the solution $v$ of
the scattering problem (\ref{1.1+})-(\ref{1.4+}) with $u^i=\nabla_x\Phi(x,z)\cdot\vec{a}$
satisfies the Lippmann-Schwinger-type equation
\be\label{2.28c}
v(x;z) =U^i(x;z) - k^2\int_{D\se\ov{D_b}}[1-n(y)]G(x,y)v(y;z)dy.
\en

Conversely, it is easy to prove that the solution of (\ref{2.28c}) also satisfies the scattering
problem (\ref{1.1+})-(\ref{1.4+}).

\begin{remark}\label{rm2.5a}{\rm
The equivalence between the scattering problem (\ref{1.1+})-(\ref{1.4+}) and the Lippmann-Schwinger-type
equation (\ref{2.28c}) still holds for a general incident field $u^i$ such as a plane wave
$u^i=e^{ikx\cdot d}$ with $d\in S^2$ and a point source $u^i=\Phi(x,y)$ with $y\in\R^3\se\ov{D}$.
}
\end{remark}

{\bf Proof of Theorem \ref{thm2.5}.} Define the volume operator $T$ in $L^p(D\se\ov{D_b})$ by
\ben
(T\varphi)(x):=k^2\int_{D\se\ov{D_b}}[1-n(y)]G(x,y)\varphi(y)dy\quad \text{in}\;D\se\ov{D_b}.
\enn
Then we have
\be\label{thm2.5a}
(I+T)v(\cdot;z)=U^i(\cdot;z)\quad\text{in}\;\;L^p(D\se\ov{D_b})
\en
since $U^i(\cdot;z)\in L^p(D\se\ov{D_b})$ for $6/5\leq p<2$. It follows from \cite[Theorem 9.9]{DT}
that $T$ is bounded from $L^p(D\se\ov{D_b})$ into $W^{2,p}(D\se\ov{D_b})$
and therefore compact in $L^p(D\se\ov{D_b})$. Thus, and by the uniqueness result for the scattering
problem (\ref{1.1+})-(\ref{1.4+}), the operator $I+T$ is of Fredholm type with index zero.
The Fredholm alternative then implies the existence of a unique solution $v$ in $L^p(D\se\ov{D_b})$
to (\ref{thm2.5a}) with the estimate
\be\label{2.29}
\|v(\cdot;z)\|_{L^p(D\se\ov{D_b})}\leq C\|U^i(\cdot;z)\|_{L^p(D\se\ov{D_b})}
\leq C(\|u^i(\cdot;z)\|_{L^p(D\se\ov{D_b})}+\|u^i(\cdot;z)\|_{L^\infty(\pa D_b)})
\en
for $6/5\leq p<2$. From this, the well-posedness of (\ref{2.28a}) and the embedding result that
$W^{2,p}(D\se\ov{D_b})\hookrightarrow H^1(D\se\ov{D_b})$ for $6/5\le p<2$,
the required estimate (\ref{2.28}) follows. The proof is thus complete.
$\Box$

\subsection{Uniqueness of the inverse problem}

Based on Lemma \ref{le2.2}, Corollary \ref{co2.1} and Theorems \ref{thm2.3} and \ref{thm2.5},
we shall prove the global uniqueness result in determining the inhomogeneous penetrable
obstacle $D$ disregarding its contents if the transmission coefficient $\la\not=1$ or if the refractive
index $n$ has a singularity at the interface $\pa D$ in the case $\la=1$, that is, $n$ satisfies
the following assumption (A).

{\bf Assumption (A)}: there exists an open neighborhood of $\pa D$, $\mathcal{O}(\pa D)\Subset D\se\ov{D_b}$,
and a positive constant $\vep_0>0$ such that $|n(x)-1|\geq\vep_0$ for a.e. $x\in\mathcal{O}(\pa D)$.

\begin{theorem}\label{thm2.4}
Given $k>0$, let $u^\infty(\wi{x};d)$ and $\wid{u}^\infty(\wi{x};d)$ be the far-field patterns of the
scattering solutions to the transmission problem $(\ref{1.1})-(\ref{1.4})$ (or the scattering problem
$(\ref{1.1+})-(\ref{1.4+})$) with respect to the penetrable obstacle $D$ with the refractive index
$n\in L^\infty(D\se\ov{D_b})$ as well as the embedded obstacle $D_b$ and the penetrable obstacle $\wid{D}$
with the refractive index $\wid{n}\in L^\infty(\wid{D}\se\ov{\wid{D}_b})$ as well as the embedded obstacle
$\wid{D}_b$, respectively. Assume that $u^\infty(\wi{x};d)=\wid{u}^\infty(\wi{x};d)$ for all $\wi{x},d\in\Sp^2$.

{\rm (i)} If $\la\not=1$, then $D=\wid{D}$.

{\rm (ii)} If $\la=1$ and $n,\wid{n}$ satisfy Assumption (A), then $D=\wid{D}$.
\end{theorem}

\begin{proof}
Assume that $D\neq\wid{D}$. Without loss of generality, choose $z^*\in\pa D\se\pa\wid{D}$ and define
\ben
z_j:=z^*+({\delta}/{j})\nu(z^*),\;\;j=1,2,\ldots
\enn
with a sufficiently small $\delta>0$ such that $z_j\in B$, where $B$ denotes
a small ball centered at $z^*$ such that $B\cap(\ov{\wid{D}\cup D_b})=\emptyset$.
See Figure \ref{two-obstacles}.

\begin{figure}[!htbp]
  \centering
  \includegraphics[width=4in,height=2.5in]{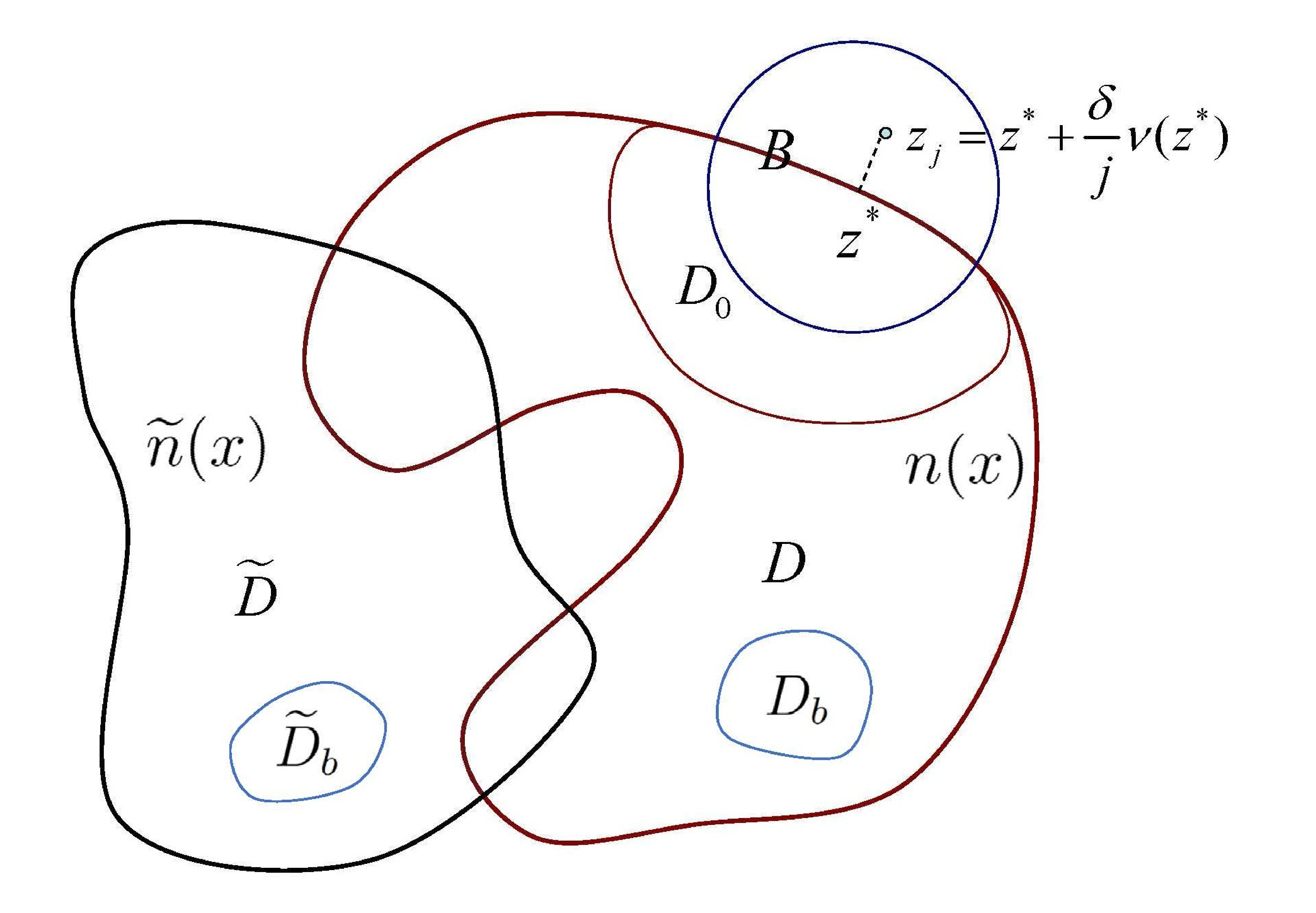}
  \caption{Two different scattering obstacles}\label{two-obstacles}
\end{figure}

(i) Let $\la\neq 1$ and let $(u_j,v_j)$, $(\wid{u_j},\wid{v_j})$ be the unique solution to
the transmission problem (\ref{1.1})-(\ref{1.4}) with respect to $D$ with refractive index
$n$ and the embedded obstacle $D_b$, $\wid{D}$ with refractive index $\wid{n}$
and the embedded obstacle $\wid{D}_b$, respectively, corresponding to the incident point source
\ben
u^i_j(x):=\Phi(x,z_j)=\frac{\exp({ik|x-z_j|})}{4\pi|x-z_j|},\qquad j=1,2,\ldots.
\enn
The assumption that $u^\infty(\wi{x};d)=\wid{u}^\infty(\wi{x};d)$ for all $\wi{x},d\in\Sp^2$,
together with Rellich's lemma and the denseness result \cite[Theorems 5.4 and 5.5]{CK}, implies that
\be\label{2.30}
u^s_j(x)=\wid{u}^s_j(x)\qquad\mbox{in}\;\;\ov{G},\qquad j=1,2,\ldots,
\en
where $G$ denotes the unbounded component of $\R^3\se\ov{(D\cup\wid{D})}$.

Since $z^*\in \pa D\se\pa\wid{D}$ and $\pa D\in C^2$, there is a small smooth ($C^2$) domain $D_0$
such that $B\cap D\subset D_0\subset D\se\ov{(\wid{D}\cup D_b)}$.
Define $U_j:=v_j$, $V_j:=\wid{u}_j$ in $D_0$. Then $(U_j,V_j)$ satisfies the
modified interior transmission problem (MITP) with $\Om=D_0$ and
\ben
&& \rho_1(j):=-(k^2n+1)v_j|_{D_0},\quad\quad \rho_2(j):=-(k^2+1)\wid{u}_j|_{D_0}, \\
&& f_1(j):=(v_j-\wid{u}_j)|_{\pa D_0},\quad\quad\quad
f_2(j):=(\la\pa v_j/\pa\nu-\pa\wid{u}_j/\pa\nu)|_{\pa D_0}.
\enn
From (\ref{2.30}) it is clear that $f_1(j)=f_2(j)=0$ on $\G_1:=\pa D_0\cap\pa D$. Since
$z^*$ has a positive distance from $\wid{D}$, the well-posedness of the transmission problem
(\ref{1.1})-(\ref{1.4}) implies that
\be\label{2.31}
\|\wid{u}_j^s\|_{H^2(D_0)}\leq C\quad\mbox{uniformly\;for}\;\;j\in\N.
\en
This implies that $\rho_2(j)\in L^2(D_0)$ is uniformly bounded for $j\in\N$
since $\Phi(\cdot,z_j)\in L^2(D_0)$ is uniformly bounded for $j\in\N$.
From Corollary \ref{co2.4} it is known that $\rho_1(j)$ is uniformly bounded in $L^2(D_0)$
for $j\in\N$.

We now prove that $f_1(j)$ and $f_2(j)$ are uniformly bounded in $H^{{1}/{2}}(\pa D_0)$
and $H^{-{1}/{2}}(\pa D_0)$, respectively, for $j\in\N$.
To this end, define $w_j:=v_j-\wid{u}_j^s-\Phi(\cdot,z_j)=v_j-\wid{u}_j$.
Then $w_j\in H^2(D\se\ov{(\wid{D}\cup D_b)})$ for every $j\in\N$ since $z_j\in\R^3\se\ov{D}$ and
is a solution to the problem
\ben
\tr w_j=g_j\quad\text{in}\;\;D\se\ov{(\wid{D}\cup D_b)},\quad\quad w_j|_{\G_1}=0,
\enn
where $g_j:=k^2[\Phi(\cdot,z_j)+\wid{u}_j^s-nv_j]\in L^2(D\se\ov{(\wid{D}\cup D_b)})$.
It follows from \cite[Theorem 9.13]{DT} that
\be\label{2.31a}
\|w_j\|_{H^2(D_0)}\le C\big(\|w_j\|_{L^2(D\se\ov{(\wid{D}\cup D_b)})}
+\|g_j\|_{L^2(D\se\ov{(\wid{D}\cup D_b)})}\big)\le C
\en
uniformly for $j\in\N$. Since $f_1(j)=w_j\big|_{\pa D_0}$ and
$$
f_2(j)=\la\pa w_j/\pa\nu+(\la-1)[\pa\wid{u}_j^s/\pa\nu+\pa\Phi(\cdot,z_j)/\pa\nu],
$$
it easily follows, by using (\ref{2.31}), (\ref{2.31a}) and the fact that
$f_1(j)|_{\G_1}=f_2(j)|_{\G_1}=0$, that $f_1(j)$ and $f_2(j)$ are uniformly
bounded in $H^{{1}/{2}}(\pa D_0)$ and $H^{-{1}/{2}}(\pa D_0)$, respectively, for $j\in\N$.
Thus, by Lemma \ref{le2.2} we have
\ben
\|\Phi(\cdot,z_j)\|_{H^1(D_0)}-\|\wid{u}^s_j\|_{H^1(D_0)}
\leq\|\wid{u}_j\|_{H^1(D_0)}=\|V_j\|_{H^1(D_0)}\leq C.
\enn
However, this is a contradiction since $\|\wid{u}^s_j\|_{H^1(D_0)}$ is uniformly
bounded and $\|\Phi(\cdot,z_j)\|_{H^1(D_0)}\rightarrow\infty$ as $j\rightarrow\infty$.
Hence, $D=\wid{D}$.

(ii) Consider the incident point source of higher-order:
\ben
u^i_j(x):=\nabla_x\Phi(x,z_j)\cdot\vec{a},\quad j=1,2,\ldots,
\enn
where $\vec{a}\in\R^3$ is a fixed vector, and let $(u_j,v_j)$ and $(\wid{u_j},\wid{v_j})$ be the
unique solution to the scattering problem (\ref{1.1+})-(\ref{1.4+}) with respect to $D$ with
refractive index $n$ and $\wid{D}$ with refractive index $\wid{n}$, respectively,
corresponding to the incident wave $u^i(x)=u^i_j(x).$
Similarly as in the proof of (i), by Rellich's lemma and the denseness
result \cite[Theorems 5.4 and 5.5]{CK}, it again follows, from the assumption
$u^\infty(\wi{x};d)=\wid{u}^\infty(\wi{x};d)$ for all $\wi{x},d\in\Sp^2$, that
\be\label{2.32}
u^s_j(x)=\wid{u}^s_j(x)\quad\text{in}\;\;\ov{G},\quad j=1,2,\ldots.
\en
Then $(U_j,V_j):=(v_j,\wid{u_j})$ satisfies the interior transmission problem (ITP) with $\Om=D_0$ and
\ben
f_1(j):=(v_j-\wid{u}_j)|_{\pa D_0},\qquad
f_2(j):=(\pa v_j/\pa\nu-\pa\wid{u}_j/\pa\nu)|_{\pa D_0}.
\enn
From (\ref{2.32}) it follows that $f_1(j)=f_2(j)=0$ in $\G_1$.

In order to utilize Corollary \ref{co2.1} to derive a contradiction,
we need to verify that $f_1(j),f_2(j)$ satisfy the condition (\textbf{C}).
To this end, we choose a cut-off function $\chi\in C^\infty_0(\R^3)$ such that
\ben
\chi(x)=\left\{\begin{array}{ll}
            0, & \text{in}\;\;\R^3\se B, \\
            1, & \text{in}\;\;B_1,
          \end{array}
        \right.
\enn
where $B_1$ is a small ball centered at $z^*$ satisfying that $B_1\subsetneq B$.
Define the function
\ben
w_0(j):=[1-\chi(x)](v_j-\wid{u}_j).
\enn
It is easy to see that $w_0(j)\in H^1_\tr(D_0)$ for every $j\in\N$
with $w_0(j)\big|_{\pa D_0}=f_1(j)$ and ${\pa w_0(j)}/{\pa\nu}\big|_{\pa D_0}=f_2(j)$.

We now prove that $\|w_0(j)\|_{H^1_\tr(D_0)}$ is uniformly bounded for $j\in\N.$
Since $z^*$ has a positive distance from $\wid{D}$, we have, by the well-posedness of
the scattering problem (\ref{1.1+})-(\ref{1.4+}), that
\be\label{2.33}
\|\wid{u}^s_j\|_{H^1_\tr(D_0)}\leq C
\en
uniformly for $j\in\N$. It further follows from Theorem \ref{thm2.5} that
\be\label{2.34}
\|v_j-u_j^i\|_{H^1(D_0)}\leq C
\en
uniformly for $j\in\N$. This, combined with (\ref{2.33}), yields that
\be\label{2.35}
\|w_0(j)\|_{H^1(D_0)}\le\|v_j-\wid{u}_j\|_{H^1(D_0)}=\|v_j-u^i_j-\wid{u}^s_j\|_{H^1(D_0)}\le C.
\en
It remains to prove that $\tr w_0(j)$ is uniformly bounded in $L^2(D_0)$.
By a direct computation, it is found that
\ben
\tr w_0(j)=\tr(1-\chi)(v_j-\wid{u}_j)+2\nabla(1-\chi)\cdot\nabla(v_j-\wid{u}_j)+(1-\chi)\tr(v_j-\wid{u}_j).
\enn
In view of (\ref{2.35}) the first and second terms are uniformly bounded in $L^2(D_0)$. Since
$\|u_j^i\|_{L^2(D_0\se\ov{B_1})}=\|\nabla\Phi(\cdot,z_j)\cdot\vec{a}\|_{L^2(D_0\se\ov{B_1})}\leq C$,
and by (\ref{2.33}) and (\ref{2.34}), we have
\ben
\tr(v_j-\wid{u}_j)=k^2[\wid{u}_j^s-n(v_j-u_j^i)]+k^2(1-n)u_j^i
\enn
is uniformly bounded in $L^2(D_0\se\ov{B_1})$ for $j\in\N.$
This, together with the fact that $\chi|_{B_1}=0$, implies that $\|\tr w_0(j)\|_{L^2(D_0)}\leq C$
uniformly for $j\in\N.$ This combined with (\ref{2.35}) gives the uniform boundedness of
$\|w_0(j)\|_{H^1_\tr(D_0)}$. Thus, $f_1(j),f_2(j)$ satisfy the condition (\textbf{C})
with $w_0(j)\in H^1_\tr(D_0)$ for every $j\in\N$.

It is well known that the smallest Dirichlet eigenvalue $\la_1(D_0)$ of $-\tr$ in $D_0$
tends to $+\infty$ as the diameter $\rho$ of $D_0$ goes to zero.
Thus, $D_0$ can be chosen such that its diameter $\rho$ is sufficiently small so
$k^2<\min\{\la_1(D_0),\la_1(D_0)/\sup(n)\}$. This, together with Corollary \ref{co2.1},
implies that $\|V_j\|_{L^2(D_0)}=\|\wid{u}_j\|_{L^2(D_0)}$ is uniformly bounded for $j\in\N.$
Thus,
\be\label{2.37}
\|\nabla_x\Phi(\cdot,z_j)\cdot\vec{a}\|_{L^2(D_0)}-\|\wid{u}_j^s\|_{L^2(D_0)}
\leq\|\wid{u}_j\|_{L^2(D_0)}\leq C
\en
uniformly for $j\in\N.$ On choosing $\vec{a}=\nu(z^*)$, it is easy to see that
\ben
\int_{D_0}|\nabla_x\Phi(\cdot,z_j)\cdot\nu(z^*)|^2dx
\ge \frac{C}{j^2}\int_{D_0}\frac{1}{|x-z_j|^6}=O(j).
\enn
This combined with (\ref{2.37}) leads to a contradiction as $j\rightarrow\infty$.
Thus, $D=\wid{D}$. The proof of the theorem is then completed.
\end{proof}

\begin{remark}\label{rm2.4}{\rm
(i) Theorem \ref{thm2.4} can be easily extended to the two-dimensional case.

(ii) In two-dimensions, the unique determination of $n$ is established in \cite{IUY10} if $D_b$ is known 
and is a sound-soft obstacle and $n\in C^{2+\al}(\ov{D\se{D_b}})$ with $\al>0$, 
whilst the unique determination of both $D_b$ and $n$ is proved in \cite{HP98} if $D_b$ is sound-soft and 
$n\in C^{1,\al}(\ov{\R^3\se\ov{D}})$ with $0<\al<1$, and $D$ can thus be uniquely determined in these two 
cases. However, the results in \cite{IUY10,HP98} do not apply to the case in this paper.
}
\end{remark}

\section{The inverse electromagnetic scattering problem}\label{sec3}
\setcounter{equation}{0}

In this section, we apply the technique introduced in Section \ref{sec2} to obtain similar
uniqueness results for inverse electromagnetic scattering by penetrable obstacles with embedded obstacles.

\subsection{Interior transmission problems}

In the case $\la_H=\mu_0/\mu_1\not=1$, consider the modified interior transmission problem (MITPM):
\be\label{3.111}
\curl\curl E_0+E_0=\xi_1 && \text{in}\;\;\Om,\\ \label{3.121}
\curl\curl F_0+F_0=\xi_2 &&  \text{in}\;\;\Om,\\ \label{3.131}
(F_0)_T-(E_0)_T=h_1\ti\nu &&  \text{on}\;\;\pa\Om,\\ \label{3.141}
\nu\ti\curl F_0-\la_H\nu\ti\curl E_0=h_2 && \text{on}\;\;\pa\Om,
\en
where $(\cdot)_T=\nu\ti(\cdot)\ti\nu$, $\xi_1,\xi_2\in L^2(\Om)^3$ and $h_1,h_2\in Y(\pa\Om)$.
Here,
\ben
Y(\pa\Om)=\{\textbf{u}\in H^{-{1}/{2}}(\pa\Om):\;\nu\cdot\textbf{u}=0,\;
\text{Div}(\textbf{u})\in H^{-{1}/{2}}(\pa\Om)\}
\enn
denotes the trace space of $H(\curl,\Om)$. For the problem (MITPM), we have the following well-posedness
result.

\begin{lemma}\label{le3.2}(\cite[Theorem 3.3]{CC0})
If $\la_H\neq 1$, then the modified interior transmission problem (MITPM) admits a unique
solution $(E_0,F_0)\in H(\curl,\Om)\times H(\curl,\Om)$ satisfying the estimate
\ben
\|E_0\|_{H(\curl,\Om)}+\|F_0\|_{H(\curl,\Om)}
\le C\big(\|\xi_1\|_{L^2(\Om)^3}+\|\xi_2\|_{L^2(\Om)^3}+\|h_1\|_{Y(\pa\Om)}+\|h_2\|_{Y(\pa\Om)}\big).
\enn
\end{lemma}

For the case $\la_H=\mu_0/\mu_1=1$, we consider the interior transmission problem (ITPM):
\be\label{3.1}
\curl\curl E_0-k^2E_0=0 && \text{in}\;\;\Om,\\ \label{3.2}
\curl\curl F_0-k^2n(x)F_0=0 && \text{in}\;\;\Om,\\ \label{3.3}
F_0\times\nu-E_0\times\nu=h_1 && \text{on}\;\;\pa\Om,\\ \label{3.4}
\curl F_0\ti\nu-\curl E_0\ti\nu=h_2 && \text{on}\;\;\pa\Om.
\en
To study the well-posedness of the above problem (ITPM), we introduce the Hilbert spaces
\ben
&&H(\curl,\Om):=\{\textbf{u}\in L^2(\Om)^3:\;\curl\textbf{u}\in L^2(\Om)^3\},\\
&&H_0(\curl,\Om):=\{\textbf{u}\in H(\curl,\Om):\;\textbf{u}\times\nu=0\}
\enn
with the inner product
$(\textbf{u},\textbf{v})_{\curl}=(\textbf{u},\textbf{v})_{L^2}+(\curl\textbf{u},\curl\textbf{v})_{L^2}$,
and the Hilbert spaces
\ben
&&\mathcal{U}(\Om):=\{\textbf{u}\in H(\curl,\Om):\;\curl\textbf{u}\in H(\curl,\Om)\},\\
&&\mathcal{U}_0(\Om):=\{\textbf{u}\in H_0(\curl,\Om):\;\curl\textbf{u}\in H_0(\curl,\Om)\}
\enn
with the inner product
$(\textbf{u},\textbf{v})_{\mathcal{U}}=(\textbf{u},\textbf{v})_{\curl}
+(\curl\textbf{u},\curl\textbf{v})_{\curl}$.

For the data $(h_1,h_2)$ we need the following assumption (\textbf{H}): the data $(h_1,h_2)$ satisfies
the property that there always exists a $\textbf{w}\in\mathcal{U}(\Om)$ such that
\be\label{3.5}
\textbf{w}\times\nu=h_1,\qquad\curl\textbf{w}\times\nu=h_2\quad\text{on}\;\;\pa\Om.
\en
Define the set $\tau(\pa\Om)$ which consists of $(h_1,h_2)$ satisfying the property (\ref{3.5})
and is equipped with the norm
\ben
\|(h_1,h_2)\|_{\tau(\pa\Om)}
:=\inf\{\|\textbf{w}\|_{\mathcal{U}(\Om)}\;:\;\textbf{w}\times\nu=h_1,\;\;
\curl\textbf{w}\times\nu=h_2\quad\text{on}\;\;\pa\Om\}.
\enn

\begin{definition}\label{def1}
If $(E_0,F_0)\in L^2(\Om)^3\times L^2(\Om)^3$ satisfies the interior transmission problem (ITPM)
in the distribution sense such that $F_0-E_0\in\mathcal{U}(\Om)$, then $(E_0,F_0)$ is called a weak
solution of the interior transmission problem (ITPM).
\end{definition}

Let $\textbf{u}:=F_0-E_0$. Then $\textbf{u}$ satisfies
\be\label{3.6}
(\curl\curl-k^2n)\frac{1}{n-1}(\curl\curl-k^2)\textbf{u}=0&&\text{in}\;\;\Om,\\ \label{3.7}
\textbf{u}\ti\nu=h_1,\;\;\curl\textbf{u}\ti\nu=h_2&&\text{on}\;\;\pa\Om.
\en
It is easy to see that the interior transmission problem (ITPM) is equivalent to the
variational problem: Find $\textbf{u}\in\mathcal{U}$ with the boundary condition (\ref{3.7})
such that
\be\label{3.8}
b(\textbf{u},\textbf{v})=0\quad\text{for\;all\;}\;\;\textbf{v}\in\mathcal{U}_0(\Om),
\en
where
\ben
b(\textbf{u},\textbf{v})=\int_{\Om}\frac{1}{n-1}(\curl\curl\textbf{u}-k^2\textbf{u})
\cdot(\curl\curl\ov{\textbf{v}}-k^2n\ov{\textbf{v}})dx.
\enn
Let $\wid{\textbf{u}}:=\textbf{u}-\textbf{w}$. Then $\wid{\textbf{u}}\in\mathcal{U}_0(\Om)$ and
(\ref{3.8}) is equivalent to the problem
\be\label{3.9}
b(\wid{\textbf{u}},\textbf{v})=-b(\textbf{w},\textbf{v})\quad\text{for\;all\;}\;\;
\textbf{v}\in\mathcal{U}_0(\Om).
\en
Based on (\ref{3.9}), the following result can be obtained (see \cite{CGH}).

\begin{lemma}\label{le3.1}(\cite[Lemma 2.9]{CGH})
If $n>1+r_0$ or $0<n<1-r_1$ with some constants $r_0,r_1>0$, then
\ben
b(\wid{\textbf{u}},\wid{\textbf{u}})\geq C\|\wid{\textbf{u}}\|^2_{\mathcal{U}(\Om)}
\enn
for $0<k^2<\min\{\la_1(\Om),\la_1(\Om)/\sup(n)\}$,
where $\la_1(\Om)$ is the smallest Dirichlet eigenvalue of $-\tr$ in $\Om$.
\end{lemma}

By Lemma \ref{le3.1} we have the following corollary.

\begin{corollary}\label{co3.1}
For any fixed $k>0$, if the diameter of $\Om$ is small enough (so $\la_1(\Om)$ is
large enough) so that $k^2<\min\{\la_1(\Om),\la_1(\Om)/\sup(n)\}$,
then the interior transmission problem (ITPM)
has a unique solution $(E_0,F_0)\in L^2(\Om)^3\times L^2(\Om)^3$ with
\be\label{3.10}
\|E_0\|_{L^2(\Om)^3}+\|F_0\|_{L^2(\Om)^3}\le C\|(h_1,h_2)\|_{\tau(\pa\Om)}.
\en
\end{corollary}

\begin{proof}
It is clear that $-b(\textbf{w},\textbf{v})$ defines a bounded, linear functional
on $\mathcal{U}_0(\Om)$. By Lemma \ref{le3.1} and the Lax-Milgram theorem, it is easy to see
that there exists a unique solution $\wid{\textbf{u}}\in\mathcal{U}_0(\Om)$ such that
\ben
\|\wid{\textbf{u}}\|_{\mathcal{U}_0(\Om)}\leq C \|\textbf{w}\|_{\mathcal{U}(\Om)},
\enn
where $C>0$ is independent of the choice of $\textbf{w}$.
This, combined with the fact that $\textbf{u}=\wid{\textbf{u}}+\textbf{w}$, gives
\ben
\|\textbf{u}\|_{\mathcal{U}(\Om)}\leq (C+1)\|\textbf{w}\|_{\mathcal{U}(\Om)},
\enn
which implies that
\ben
\|\textbf{u}\|_{\mathcal{U}(\Om)}\le (C+1)\inf_{\textbf{w}}\|\textbf{w}\|_{\mathcal{U}(\Om)}
=(C+1)\|(h_1,h_2)\|_{\tau(\pa\Om)}.
\enn
Define $F_0:=[{1}/({n-1})](\curl\curl-k^2)\textbf{u}$ and $E_0:=F_0-\textbf{u}$. Then
$(E_0,F_0)\in L^2(\Om)^3\times L^2(\Om)^3$ and satisfies the the interior transmission problem (ITPM)
with the estimate (\ref{3.10}).
\end{proof}

\begin{remark}\label{rm3.1}{\rm
Similarly as in the acoustic case, our uniqueness result for the case $\la_H=\mu_0/\mu_1\not=1$ is
associated with the well-posed problem (MITPM), while the uniqueness result for the case $\la_H=\mu_0/\mu_1=1$
is associated with the much harder problem (ITPM) which may not be well-posed for a given wavenumber $k$
if $\Om$ is not small enough (see Lemma \ref{le3.1} and Corollary \ref{co3.1}). This is the reason why
all the previous methods for the uniqueness proofs of the inverse problems require the assumption that
$\la_H=\mu_0/\mu_1\not=1$.
}
\end{remark}

\subsection{A priori estimates for the forward scattering problem with $L^p$ data}

We now establish a priori estimates of solutions to the electromagnetic transmission problem
(\ref{1.6})-(\ref{1.9}) or the electromagnetic scattering problem (\ref{1.6+})-(\ref{1.9+})
with $L^p$ data $(1<p<2)$, which are needed in the uniqueness proofs.

Introduce the magnetic dipole operator $M_{\rm{e}\rm{e}}$ and the electric dipole operator $N_{\rm{e}\rm{e}}$ by
\ben
(M_{\rm{e}\rm{e}}a)(x)&=&\int_{\pa D}\nu(x)\times\curl_x\{a(y)\Phi(x,y)\}ds(y),\quad x\in\pa D,\\
(N_{\rm{e}\rm{e}}b)(x)&=&\nu(x)\times\curl\curl\int_{\pa D}\nu(y)\times b(y)\Phi(x,y)ds(y),\quad x\in\pa D.
\enn
Similarly, we also introduce the operators $M_{1,\rm{e}\rm{e}}$ and $N_{1,\rm{e}\rm{e}}$
which are defined as $M_{\rm{e}\rm{e}}$ and $N_{\rm{e}\rm{e}}$ with
the kernel $\Phi(x,y)$ replaced by $\Phi_1(x,y)$ as well as the operators $\mathscr{M}_{\rm{t}\rm{h}}$
for ${\rm t,h=e,i}$, respectively, and $\mathscr{M}=M, M_1,N$ and $N_1$ which are defined, for example,
as $M_{\rm{e}\rm{i}}$ on $\pa D$ but with $x\in\pa D_b$.

\begin{theorem}\label{thm3.1}
Assume that $\la_H=\mu_0/\mu_1\not=1$ and $n\in L^\infty(D\se\ov{D_b})$.
For $z^*\in\pa D$ let $B_{z^*}$ be a small ball centered at $z^*$.
Let $z\in B_{z^*}\cap(\R^3\se\ov{D})$ and let $(E,G)$ be the solution of the transmission problem
$(\ref{1.6})-(\ref{1.9})$ corresponding to the incident magnetic dipole
$E^i(x)=\curl(p\Phi(x,z))/\|\curl(p\Phi(x,z))\|_{L^2(\pa D)}$.
Then $G\in L^2(D\se\ov{D_b})^3\cap H(\curl,D\se\ov{D_b\cup B_{z^*}})$ with
\be\label{3.151}
\|G\|_{L^2(D\se\ov{D_b})^3}+\|G\|_{H(\curl,D\se\ov{D_b\cup B_{z^*}})}\leq C,
\en
where $C>0$ is a constant and independent of $z$.
\end{theorem}

\begin{proof}
To prove (\ref{3.151}), define $G_1(x):=G(x)-({\mu_1}/{\mu_0})E_1^i(x)$ for $x\in D\se\ov{D_b}$
with $E_1^i(x):=\curl(p\Phi_1(x,z))/\|\curl(p\Phi_1(x,z))\|_{L^2(\pa D)}$ with $k_1^2\neq k^2$.
Then $E^s|_{\R^3\se\ov{D}}$ and $G_1|_D$ satisfy the transmission problem (TP1):
\be\label{3.161}
\curl\curl E^s-k^2E^s=0 && \text{in}\;\;\R^3\se\ov{D},\\ \label{3.171}
\curl\curl G_1-k^2n(x)G_1=g&& \text{in}\;\;D\se\ov{D_b},\\ \label{3.181}
\nu\times E^s-\nu\times G_1=f_1&& \text{on}\;\;\pa D,\\ \label{3.191}
\nu\times\curl E^s-\frac{\mu_0}{\mu_1}\nu\times\curl G_1=f_2&& \text{on}\;\;\pa D,\\ \label{3.191a}
\mathscr{B}_E(G_1)=f_3 && \text{on}\;\;\pa D_b,
\en
with the radiation condition (\ref{1.9}), where
$f_1:=({\mu_1}/{\mu_0})\nu\times E_1^i|_{\pa D}-\nu\times E^i|_{\pa D}$,
$f_2:=\nu\times\curl E_1^i|_{\pa D}-\nu\times\curl E^i|_{\pa D}$,
$f_3:=-({\mu_0}/{\mu_1})\mathscr{B}_E(E^{\rm i}_1)$ and $g=({\mu_1}/{\mu_0})(k^2n-k_1^2)E^i_1$.
For convenience, we only consider a perfect conductor condition on $\pa D_b$, i.e.,
$\mathscr{B}_E(G_1)=\nu\times G_1$. The same result can be similarly extended to an impedance
condition on $\pa D_b$.

Step 1. We first consider the case $k^2n(x)\equiv k_1^2$. In this case, $g=0$, and
we seek the solution of the problem (TP1) in the form
\ben
E^s(x)&=&\mu_0\curl\int_{\pa D}a(y)\Phi(x,y)ds(y)
      +\curl\curl\int_{\pa D}b(y)\Phi(x,y)ds(y),\quad x\in\R^3\se\ov{D},\\
G_1(x)&=&\mu_1\curl\int_{\pa D}a(y)\Phi_1(x,y)ds(y)
      +\curl\curl\int_{\pa D}b(y)\Phi_1(x,y)ds(y)\\
      &&+\curl\int_{\pa D_b}c(y)\Phi_1(x,y)ds(y)
           +{\rm i}\curl\curl\int_{\pa D_b} (RS_0^2 c)(y)\Phi_1(x,y)ds(y),\; x\in D\se\ov{D_b},
\enn
with tangential fields $a$ and $b$. Here, $Rh:=(\nu\times h)\times \nu$ stands for the projection of
a vector defined on $\pa D_b$ onto the tangent plane and $S_0$ denotes the single-potential operator
defined on $\pa D_b$ by $S_{\rm{i}\rm{i}}$ with the wave number $k=0$.
Then the problem (TP1) is equivalent to the following system of integral equations:
\be\label{3.201}
\frac{1}{2}\ov{\mu}a+(\mu_0M_{\rm{e}\rm{e}}-\mu_1M_{1,\rm{e}\rm{e}})a
 +(N_{\rm{e}\rm{e}}-N_{1,\rm{e}\rm{e}})Pb -A_{\rm{i}\rm{e}}c=f_1,\\ \label{3.211}
\frac{1}{2}\wid{\mu}b+\mu_0(N_{\rm{e}\rm{e}}-N_{1,\rm{e}\rm{e}})Pa
  +(k^2M_{\rm{e}\rm{e}}-\frac{\mu_0}{\mu_1}k^2_1M_{1,\rm{e}\rm{e}})b-B_{\rm{i}\rm{e}}c=f_2,\\ \label{3.211a}
\frac{1}{2}\mu_0 c+\mu_0(M_{1,\rm{i}\rm{i}}+{\rm i}N_{1,\rm{i}\rm{i}}PRS^2_0)c
  +\mu_0\mu_1 M_{1,\rm{e}\rm{i}}a +\mu_0 N_{1,\rm{e}\rm{i}}Pb=f_3,
\en
with $Ph:=h\times \nu$, $\wid{\mu}:=(k^2_1\mu_0/\mu_1+k^2)$, $\ov{\mu}:=(\mu_0+\mu_1)$,
$A_{\rm{i}\rm{e}}:=(M_{1,\rm{i}\rm{e}}+{\rm i}N_{1,\rm{i}\rm{e}}PS^2_0)$ and
$B_{\rm{i}\rm{e}}:=({\mu_0}/{\mu_1})(N_{1,\rm{i}\rm{e}}P+{\rm i}k^2_1M_{1,\rm{i}\rm{e}}RS^2_0)$.
It is easy to see that $f_1\in T^2(\pa D)$ and $f_2\in T^2(\pa D)$
and $f_3\in C^{0,\al}_d(\pa D_b)$ with $\al\in (0,1)$.
Further, from the identity $\dive_{\pa D}(\nu\times\curl E^i)=-\nu\cdot\curl^2 E^i$ it follows
that $f_2\in T^2_d(\pa D)$. Here,
\ben
T^2(\pa D)&:=&\{\textbf{u}\in L^2(\pa D)^2:\;\nu\cdot\textbf{u}=0\},\\
T^2_d(\pa D)&:=&\{\textbf{u}\in L^2(\pa D)^2:\;\nu\cdot\textbf{u}=0,\;
      \dive_{\pa D}\textbf{u}\in L^2(\pa D)\},\\
C^{0,\al}_d(\pa D_b)&:=&\{\textbf{u}\in C^{0,\al}(\pa D_b):\;\nu\cdot\textbf{u}=0,\;
                                  \dive_{\pa D_b}\textbf{u}\in C^{0,\al}(\pa D_b)\}.
\enn
Since $\pa D\in C^2$, it follows from \cite{K89} or \cite{P0} that
the system (\ref{3.201})-(\ref{3.211a}) is of Fredholm type in the space
$T^2(\pa D)\times T^2_d(\pa D)\times C^{0,\al}_d(\pa D_b)$. This, together with the uniqueness
of the transmission problem (TP1), implies that the system (\ref{3.201})-(\ref{3.211}) has a unique solution
$(a,b,c)^T\in T^2(\pa D)\times T^2_d(\pa D)\times C^{0,\al}_d(\pa D_b)$ with the estimate
\be\label{3.221}
\|a\|_{T^2(\pa D)}+\|b\|_{T^2_d(\pa D)}+\|c\|_{C^{0,\al}_d(\pa D_b)}
\leq C(\|f_1\|_{T^2(\pa D)}+\|f_2\|_{T^2_d(\pa D)}+\|f_3\|_{C^{0,\al}_d(\pa D_b)}).
\en
We now split $G_1$ into three parts $G_1^{(1)}$, $G_1^{(2)}$ and $G_1^{(3)}$, given by
\ben
G_1^{(1)}(x)&=&\curl\int_{\pa D}a(y)\Phi_1(x,y)ds(y),\quad x\in D\se\ov{D_b},\\
G_1^{(2)}(x)&=&\curl\curl\int_{\pa D}b(y)\Phi_1(x,y)ds(y),\quad x\in D\se\ov{D_b},
\enn
and $G_1^{(3)}:=G_1-\mu_1G_1^{(1)}-G_1^{(2)}$ in $D\se\ov{D_b}$.
Then $G_1=\mu_1G_1^{(1)}+G_1^{(2)}+G_1^{(3)}$. Moreover, by the properties of
$M_{1,\rm{t}\rm{t}}$ and $N_{1,\rm{t}\rm{t}}$ for ${\rm t=e,i}$, we have
\ben
&&\curl\curl G_1^{(1)}-k^2_1G_1^{(1)}=0\;\;\text{in}\;\;D,\quad\nu\times G_1^{(1)}\in T^2(\pa D),\\
&&\curl\curl G_1^{(2)}-k^2_1G_1^{(2)}=0\;\;\text{in}\;\;D,\quad\nu\times G_1^{(2)}\in T^2_d(\pa D),\\
&&\curl\curl G_1^{(3)}-k^2_1G_1^{(3)}=0\;\;\text{in}\;\;\R^3\se\ov{D_b},\quad
    \nu\times G_1^{(3)}\in C^{0,\al}_d(\pa D_b)\subset T^2_d(\pa D).
\enn
It is easy to see that $G_1^{(1)}\in L^2(D)$, $G_1^{(2)}\in H(\curl,D)$,
$G_1^{(3)}\in H_{\rm loc}(\curl,\R^3\se\ov{D_b})$ and
\be\label{3.231}
\|G_1^{(1)}\|_{L^2(D)^3}+\|G_1^{(2)}+G_1^{(3)}\|_{H(\curl,D\se\ov{D_b})}
\leq C(\|a\|_{T^2(\pa D)}+\|b\|_{T^2_d(\pa D)}+\|c\|_{C^{0,\al}_d(\pa D_b)}).
\en
Note that $f_1\in T^2_d(\pa D\se\ov{B_{z^*}})$. This, together with (\ref{3.201}) and the fact
that $M_{\rm{e}\rm{e}},M_{1,\rm{e}\rm{e}}$ and $(N_{\rm{e}\rm{e}}-N_{1,\rm{e}\rm{e}})P$
are bounded from $T^2(\pa D)$ and $T^2_d(\pa D)$ into $T^2_d(\pa D)$,
gives that $a\in T^2_d(\pa D\se\ov{B_{z^*}})\cap T^2(\pa D)$.
Choose a ball $B_1$ centered at $z^*$ with $B_1\subsetneq B_{z^*}$ and a cut-off function
$\chi\in C^\infty_0(\R^3)$ supported in $B_1$ with $\chi=1$ in a small neighborhood of $z^*$.
Then $G_1^{(1)}$ can be written in the form
\ben
G_1^{(1)}(x)=\curl\int_{\pa D}\Phi_1(x,y)[1-\chi(y)]a(y)ds
    +\curl\int_{\pa D}\Phi_1(x,y)\chi(y)a(y)ds.
\enn
Obviously, the first term belongs to $H(\curl,D)$ since $(1-\chi(y))a(y)\in T^2_d(\pa D)$,
and the second term belongs to $H(\curl,D\se\ov{B_1})$ since $B_1\subsetneq B_{z^*}$.
Combining (\ref{3.231}) and (\ref{3.221}) gives the required estimate (\ref{3.151})
in the case $k^2n(x)\equiv k_1^2$.

Step 2. For the general case $n\in L^\infty(D\se\ov{D_b})$, we consider the transmission problem (TP2):
\be\label{3.241}
\curl\curl W-k^2W=0 &&\text{in}\;\;\R^3\se\ov{D},\\ \label{3.251}
\curl\curl W_1-k^2n(x)W_1=g_1 &&\text{in}\;\;D\se\ov{D_b},\\ \label{3.261}
\nu\times W-\nu\times W_1=0 &&\text{on}\;\;\pa D,\\ \label{3.271}
\nu\times\curl W-\frac{\mu_0}{\mu_1}\nu\times\curl W_1=0 &&\text{on}\;\;\pa D\\ \label{3.271a}
\mathscr{B}_E(W_1)=0 && \text{on}\;\;\pa D_b
\en
with $W$ satisfying the radiation condition (\ref{1.9}),
where $g_1=g+(k^2n-k_1^2)\wid{G}_1\in L^2(D)^3$ and $(\wid{E}^s,\wid{G}_1)$ is a solution of
the transmission problem (IT1) with $k^2n(x)\equiv k_1^2$. By Step 1 we have
\be\label{3.151+}
\|\wid{G}_1\|_{L^2(D\se\ov{D_b})^3}+\|\wid{G}_1\|_{H(\curl,D\se\ov{D_b\cup B_{z^*}})}\leq C,
\en
where $C>0$ is a constant and independent of $z$.
By the variational method, it is easy to show that the problem (TP2) admits a unique solution
$(W,W_1)\in H_{\rm loc}(\curl,\R^3\se\ov{D})\times H(\curl,D\se\ov{D_1})$ with
\be\label{3.281}
\|W_1\|_{H(\curl,D\se\ov{D_b})}\leq C\|g_1\|_{L^2(D\se\ov{D_b})^3}
\le C(\|g\|_{L^2(D\se\ov{D_1})^3}+\|\wid{G}_1\|_{L^2(D\se\ov{D_b})^3}).
\en
Define $E^s|_{\R^3\se\ov{D}}:=\wid{E^s}+W$ and $G_1|_{D\se\ov{D_b}}:=\wid{G_1}+W_1$.
Then it is easy to check that $(E^s,G_1)$ is the unique solution of the problem (TP1).
The required estimate (\ref{3.151}) thus follows from (\ref{3.151+}) and (\ref{3.281}).
This completes the proof of the theorem.
\end{proof}

\begin{theorem}\label{thm3.2}
Assume that $\la_H=\mu_0/\mu_1=1$ and $n\in C^1(\ov{D})$. For $z^*\in\pa D$ let $B_{z^*}$ be a small
ball centered at $z^*$. Let $z\in B_{z^*}\cap(\R^3\se\ov{D})$ and
Let $(E,G)$ be the solution of the scattering problem $(\ref{1.6+})-(\ref{1.9+})$
corresponding to the incident magnetic dipole $E^i(x)=\curl(p\Phi(x,z))$.
Then $G\in L^p(D)$, $G^s:=G-E^i\in H(\curl,D)$ and
\ben
\|G\|_{L^p(D)}+\|G^s\|_{H(\curl,D)}\leq C\|\curl(p\Phi(x,z))\|_{L^p(D)}
\enn
for $6/5\leq p<3/2$, where $C>0$ is independent of $z$.
\end{theorem}

\begin{proof}
Since $\curl(p\Phi(x,z))$ satisfies the Maxwell equation $\curl\curl E^i-k^2E^i=0$ in $D$,
it follows from the Stratton-Chu formula \cite[Theorems 6.1 and 6.7]{CK} (cf. the proof of
Theorem 9.1 of \cite{CK}) that the scattering problem (\ref{1.6+})-(\ref{1.9+}) is equivalent
to the integral equation
\be\no
G(x)&=&E^i(x)-k^2\int_{D}\Phi(x,y)m(y)G(y)dy\\ \label{3.11}
    &&+\text{grad}\int_{D}\frac{1}{n(y)}\text{grad}n(y)\cdot G(y)\Phi(x,y)dy,\quad x\in D,
\en
where $m(y):=1-n(y)$. On the space $L^p(D)$, define the operators $T_1$ and $T_2$ by
\ben
(T_1\varphi)(x)&:=&k^2\int_{D}\Phi(x,y)m(y)\varphi(y)dy,\quad x\in D,\\
(T_2\varphi)(x)&:=&\text{grad}\int_{D}\frac{1}{n(y)}\text{grad}n(y)\cdot\varphi(y)\Phi(x,y)dy,
   \quad x\in D.
\enn
Then the integral equation (\ref{3.11}) can be rewritten in the form
\be\label{3.11+}
(I+T_1-T_2)G=E^i\quad\text{in}\;\;D.
\en
It follows from \cite[Theorem 9.9]{DT} that $T_1$ is bounded from $L^p(D)$ into $W^{2,p}(D)$
and $T_2$ is bounded from $L^p(D)$ into $W^{1,p}(D)$. Therefore, $T_1$ and $T_2$ are compact
operators in $L^p(D)$. This, together with the uniqueness of the scattering problem
(\ref{1.6+})-(\ref{1.9+}), implies that the operator $I+T_1-T_2$ is of Fredholm type with index zero.
The Fredholm alternative gives that the integral equation (\ref{3.11}) has a unique solution
$G\in L^p(D)$ with
\be\label{3.12}
\|G\|_{L^p(D)}\leq C\|E^i\|_{L^p(D)}.
\en
From (\ref{3.11}) it is easily seen that
\ben
G^s=(T_2-T_1)G\in W^{1,p}(D)\hookrightarrow L^2(D)
\enn
since $6/5\le p<3/2$. Further, by the identity $\curl\text{grad}=0$, we have
\ben
\curl G^s=-\curl(T_1G)\in W^{1,p}(D)\hookrightarrow L^2(D).
\enn
Thus we deduce that
\be\label{3.13}
\|G^s\|_{H(\curl,D)}\leq C\|G\|_{L^p(D)}\leq C\|E^i\|_{L^p(D)}.
\en
The proof is then completed by combining (\ref{3.12}) and (\ref{3.13}).
\end{proof}

\subsection{Determining the penetrable obstacle}

We first make use of Lemma \ref{le3.2} and Theorem \ref{thm3.1} to establish the global uniqueness result in
the inverse electromagnetic scattering problem of determining the penetrable obstacle $D$ under the assumption
that $\la_H=\mu_0/\mu_1\not=1$, disregarding its contents $n$ and $D_b$.

\begin{theorem}\label{thm3.3}
Assume that $\la_H=\mu_0/\mu_1\not=1$.
Given $k>0$, let $E^\infty(\wi{x};d;p)$, $\wid{E}^\infty(\wi{x};d;p)$ be the electric far-field
patterns with respect to the transmission scattering problem $(\ref{1.6})-(\ref{1.9})$ corresponding
to the penetrable obstacle $D$ with the refractive index $n\in L^\infty(D\se\ov{D_b})$ and the embedded obstacle
$D_b$, the penetrable obstacle $\wid{D}$ with the refractive index $\wid{n}\in L^\infty(\wid{D}\se\ov{\wid{D_b}})$,
respectively. If $E^\infty(\wi{x};d;p)=\wid{E}^\infty(\wi{x};d;p)$ for all $\wi{x},d\in\Sp^2$ and
all polarizations $p\in\R^3$, then $D=\wid{D}$.
\end{theorem}

\begin{proof}
Assume that $D\neq\wid{D}$. Without loss of generality,
choose $z^*\in\pa D\se\pa\wid{D}$ and define
\ben
z_j:=z^*+({\delta}/{j})\nu(z^*),\;\;j=1,2,\ldots
\enn
with a sufficiently small $\delta>0$ such that $z_j\in B$, where $B$ denotes
a small ball centered at $z^*$ and satisfies that $B\cap\ov{(\wid{D}\cup D_b)}=\emptyset$.
See Figure \ref{two-obstacles}.

It is easy to see that the transmission problem (\ref{1.6})-(\ref{1.9}) can be reformulated
in terms of the electric fields $E$ and $G$ as:
\be\label{3.14}
\curl\curl E-k^2E=0&&\text{in}\;\;\R^3\se\ov{D},\\ \label{3.15}
\curl\curl G-k^2n(x)G=0&&\text{in}\;\;D\se\ov{D_b},\\ \label{3.16}
\nu\ti E=\nu\ti G,\;\;\nu\ti\curl E=\la_H\nu\ti\curl G &&\text{on}\;\;\pa D,\\ \label{3.17}
\mathscr{B}_E (G)=0 && \mbox{on}\;\;\pa D_b
\en
with $E=E^i+E^s$ in $\R^3\se\ov{D}$ and the Silver-M\"{u}ller radiation condition
\be\label{3.18}
\lim_{|x|\rightarrow\infty}|x|[\curl E^s(x)\ti\wi{x}-ikE^s(x)]=0.
\en

Consider the incident magnetic dipole wave in the form
\ben
E^i_j(x)=\frac{\curl(p\Phi(x,z_j))}{\|\curl(p\Phi(x,z_j))\|_{L^2(\pa D)}}\qquad j=1,2,3,\ldots
\enn
with polarization $p\in\R^3$. Let $(E_j,G_j)$ and $(\wid{E_j},\wid{G_j})$ be the unique solution
of the transmission problem (\ref{3.14})-(\ref{3.18}) with respect to $D$ with refractive index
$n$, the embedded obstacle $D_b$ and to $\wid{D}$ with refractive index $\wid{n}$,
the embedded obstacle $\wid{D_b}$, respectively, corresponding to the incident magnetic dipole $E^i_j(x)$.
From the assumption $E^\infty(\wi{x};d;p)=\wid{E}^\infty(\wi{x};d;p)$ for
all $\wi{x},d\in\Sp^2$ and all polarizations $p\in\R^3$, and by using
Rellich's lemma and the denseness results, it follows that
\be\label{3.371}
E_j^s(x)=\wid{E_j^s}(x)\quad\text{in}\;\;\ov{G_0},
\en
where $G_0$ denotes the unbounded component of $\R^3\se\ov{(D\cup\wid{D})}$.

Since $z^*\in\pa D\se\wid{D}$ and $\pa D\in C^2$, we can choose a small $C^2$-smooth domain $D_0$
such that $B\cap D\subset D_0\subset D\se\ov{(\wid{D}\cup D_b)}$.
Let $E_0(j)=G_j$ and $F_0(j)=\wid{E_j}$ in $D_0$. Then $E_0(j)$ and $F_0(j)$ satisfy
the modified interior transmission problem (MITPM) with $\Om=D_0$ and
\ben
&&\xi_1(j):=(k^2n+1)G_j,\\
&&\xi_2(j):=(k^2+1)\wid{E_j},\\
&&h_1(j):=\nu\times\wid{E_j}-\nu\times G_j,\\
&&h_2(j):=\nu\times\curl\wid{E_j}-\frac{\mu_0}{\mu_1}\nu\times\curl G_j.
\enn
From (\ref{3.371}) and the transmission conditions on $\pa D$,
it is easily seen that $h_1(j)=h_2(j)=0$ on $\G_1:=\pa D_0\cap\pa D$.
Further, from the well-posedness of the transmission problem (\ref{1.6})-(\ref{1.9}), and
in view of the positive distance from $z^*$ to $\wid{D}$, we know that
\be\label{3.381}
\|\wid{E_j^s}\|_{H(\curl,D_0)}\leq C,
\en
where $C$ is independent of $j\in\N.$
Noting that $\|E^i_j\|_{L^2(D_0)}$ is uniformly bounded for all $j\in\N$, and by (\ref{3.381}),
we deduce that $\xi_2(j)$ is bounded in $L^2(D_0)$ uniformly for all $j\in\N.$
Moreover, by Theorem \ref{thm3.1} it follows that
\ben
\|G_j\|_{L^2(D_0)^3}+\|G_j\|_{H(\curl,D_0\se\ov{B_{z^*}})}\le C,
\enn
uniformly for all $j\in\N$, where $B_{z^*}$ is chosen such that $B_{z^*}\cap(D)\subsetneq D_0$.
Thus, $\xi_1(j)$ is bounded in $L^2(D_0)$ uniformly for $j\in\N$, and by (\ref{3.381})
$\wid{E_j}-G_j$ and $\wid{E_j}-\la_HG_j$ are bounded in $H(\curl,D_0\se\ov{B_{z^*}})$
uniformly for $j\in\N$. Then, by the trace theorem and the fact that $h_1(j)=h_2(j)=0$ on $\G_1$,
we have
\ben
\|h_1(j)\|_{Y(\pa D_0)}+\|h_2(j)\|_{Y(\pa D_0)}\leq C
\enn
uniformly for $j\in\N$.
Then, by Lemma \ref{le3.2} it is derived that
\ben
\|E_j^i\|_{H(\curl,D_0)}-\|\wid{E_j^s}\|_{H(\curl,D_0)}\le\|\wid{E_j}\|_{H(\curl,D_0)}
=\|F_0(j)\|_{H(\curl,D_0)}\le C.
\enn
Choosing $p=\nu(z^*)$ we have
\ben
\|\curl E_j^i\|_{L^2(D_0)}
&=&\frac{\|\curl\curl[\nu(z^*)\Phi(x,z_j)]\|_{L^2(D_0)}}{\|\curl[\nu(z^*)\Phi(x,z_j)]\|_{L^2(\pa D_0)}}\\
&\ge&\frac{\|\nabla\dive[\nu(z^*)\Phi(x,z_j)]\|_{L^2(D_0)}}{\|\curl[\nu(z^*)\Phi(x,z_j)]\|_{L^2(\pa D_0)}}
-\frac{\|k^2[\nu(z^*)\Phi(x,z_j)]\|_{L^2(D_0)}}{\|\curl[\nu(z^*)\Phi(x,z_j)]\|_{L^2(\pa D_0)}}:=I+II.
\enn
Obviously, the second term $II$ is uniformly bounded due to the boundedness of $\Phi(x,z_j)$
in $L^2(D_0)$. Without loss of generality, we assume that $z^*=(0,0,0)$ and $\nu(z^*)=(0,0,1)$
(in fact, other cases can be easily transformed into this one by a linear transformation).
Then we have
\ben
I^2=\int_{D_0}\frac{x_1^2+x_2^2}{|x-z_j|^8}dx\Big/\int_{\pa D_0}\frac{x_1^2+x_2^2}{|x-z_j|^6}ds(x)
=O(j^3),\qquad j\rightarrow\infty.
\enn
This implies that $\|E_j^i\|_{H(\curl,D_0)}\rightarrow\infty$ as $j\rightarrow\infty$.
However, this is a contradiction since $\|\wid{E_j^s}\|_{H(\curl,D_0)}$ is uniformly
bounded for all $j\in\N$. Therefore, we have $D=\wid{D}$.
\end{proof}

\begin{remark}\label{rm3.2}{\rm
The method can be applied to improve the uniqueness result in \cite{HP0} by relaxing
the smoothness requirement on the boundary $\pa D$ and the refractive index $n$
($\pa D\in C^2$ instead of $\pa D\in C^{2,\al}$ and $n\in C^1(\ov{D})$ instead of
$n\in C^{1,\al}(\ov{D})$, where $0<\al<1$) and by removing the assumption that
the refractive index $n$ is a constant near the boundary $\pa D$
and $\I(n(x_0))>0$ for some $x_0\in D$.
}
\end{remark}

\subsection{Determining the support of the inhomogeneous medium}

We now use Corollary \ref{co3.1} and Theorem \ref{thm3.2} to prove uniqueness in determining the support
$D$ of the inhomogeneous medium disregarding its contents provided the refractive index $n$ is differentiable
in $D$ and satisfies Assumption (A).

\begin{theorem}\label{thm3.4}
Assume that $n\in C^1(\ov{D})$ satisfies Assumption (A).
Given $k>0$, let $E^\infty(\wi{x};d;p)$ and $\wid{E}^\infty(\wi{x};d;p)$ be the electric far-field
patterns with respect to the scattering problem $(\ref{1.6+})-(\ref{1.9+})$ with the refractive indices
$n\in C^1(\ov{D})$ and $\wid{n}\in C^1(\ov{\wid{D}})$, respectively.
If $E^\infty(\wi{x};d;p)=\wid{E}^\infty(\wi{x};d;p)$ for all $\wi{x},d\in\Sp^2$ and
all polarizations $p\in\R^3$, then $D=\wid{D}$.
\end{theorem}

\begin{proof}
Assume that $D\neq\wid{D}$. Without loss of generality,
choose $z^*\in\pa D\se\pa\wid{D}$ and define
\ben
z_j:=z^*+({\delta}/{j})\nu(z^*),\;\;j=1,2,\ldots
\enn
with a sufficiently small $\delta>0$ such that $z_j\in B$, where $B$ denotes
a small ball centered at $z^*$ and satisfies that $B\cap\ov{\wid{D}}=\emptyset$
(cf. Figure \ref{two-obstacles}).

The scattering problem (\ref{1.6+})-(\ref{1.9+}) can be reformulated in terms of the electric field $E$ as:
\be\label{3.14+}
\curl\curl E-k^2n(x)E=0&&\text{in}\;\;\R^3,\\ \label{3.18+}
\lim_{|x|\rightarrow\infty}|x|[\curl E^s(x)\ti\wi{x}-ikE^s(x)]=0&&
\en
with $E=E^i+E^s$ in $\R^3$.

Consider the incident magnetic dipole wave
\ben
E^i_j(x)=\curl(p\Phi(x,z_j))\qquad j=1,2,3,\ldots,
\enn
with $p\in\R^3$. Let $E_j$ and $\wid{E_j}$ denote the unique solution to the scattering
problem (\ref{3.14+})-(\ref{3.18+}) with respect to the refractive indices $n$ and $\wid{n}$, respectively,
corresponding to the incident magnetic dipole $E^i_j(x)$.
Similarly as in the proof of Theorem \ref{thm3.3}, we have, from the assumption
$E^\infty(\wi{x};d;p)=\wid{E}^\infty(\wi{x};d;p)$ for all $\wi{x},d\in\Sp^2$ and all polarizations
$p\in\R^3$, that
\be\label{3.19}
E_j^s(x)=\wid{E_j^s}(x)\qquad\text{in}\;\;\ov{G_0},
\en
where $G_0$ denotes the unbounded component of $\R^3\se\ov{(D\cup\wid{D})}$.

Since $z^*\in\pa D\se\wid{D}$ and $\pa D\in C^2$, we can choose a small $C^2$-smooth domain $D_0$
such that $B\cap D\subset D_0\subset D\se\ov{\wid{D}}$.
Define $E_0(j):=\wid{E_j}$ and $F_0(j):=G_j$ in $D_0$. Then $(E_0(j),F_0(j))$ satisfies
the interior transmission problem (ITPM) with $\Om=D_0$ and
\ben
&&h_1(j):=\nu\times G_j-\nu\times\wid{E_j}\\
&&h_2(j):=\nu\times\curl G_j-\nu\times\curl\wid{E_j}.
\enn
From (\ref{3.19}) we see that $h_1(j)=h_2(j)=0$ in $\G_1:=\pa D_0\cap\pa D$.
Choose the cut-off function $\chi\in C^\infty_0(\R^3)$ such that
\ben
\chi(x)=\left\{
  \begin{array}{ll}
    0, & x\in\R^3\se\ov{B}, \\
    1, & x\in B_1.
  \end{array}
\right.
\enn
Here, $B_1\varsubsetneq B$ is a small ball centered at $z^*$.
Define $W(j):=[1-\chi](G_j-\wid{E_j})$.
Then it is easy to check that $\nu\times W(j)|_{\pa D_0}=h_1(j)$ and
$\nu\times\curl W(j)|_{\pa D_0}=h_2(j)$.

Since $z^*$ has a positive distance from $\wid{D}$, by the well-posedness of the scattering
problem (\ref{1.6+})-(\ref{1.9+}) (or (\ref{3.14+})-(\ref{3.18+})) we have
\be\label{3.20}
\|\wid{E_j^s}\|_{\mathcal{U}(D_0)}\leq C.
\en
On the other hand, by Theorem \ref{thm3.2} it follows that
\be\label{3.21}
\|G_j-E^i_j\|_{H(\curl,D_0)}+\|G\|_{L^p(D)}\le C\|\curl(p\Phi(x,z))\|_{L^p(D)}\le C_1
\en
for $6/5\le p<3/2$, where $C,C_1$ are independent of $j\in\N$.
This, combined with (\ref{3.20}), yields that
\be\label{3.22}
\|W(j)\|_{H(\curl,D_0)}\leq C.
\en
It remains to prove that $\curl\curl W(j)$ is uniformly bounded in $L^2(D)^3$ for $j\in\N$.
By a direct calculation, we find that
\ben
\curl\curl W(j)&=&\nabla a(x)\times\curl(G_j-\wid{E_j})+a(x)\curl\curl(G_j-\wid{E_j})\\
     &&+\curl[\nabla a(x)\times(G_j-\wid{E_j})]
\enn
with $a(x):=1-\chi(x)$.
From (\ref{3.20}) and (\ref{3.21}) it is seen that the first term
$\nabla a\times\curl(G_j-\wid{E_j})\in L^2(D_0)$.
From the Maxwell equations, it is found that
\be\label{3.23}
a(x)\curl\curl(G_j-\wid{E_j})=k^2a[n(G_j-E^i_j)-\wid{E_j^s}+(n-1)E_j^i]\in L^2(D_0)
\en
since $a|_{B_1}=0$ and $\|aE_j^i\|_{L^2(D_0)^3}\leq C$.
Further, we have
\ben
\curl[\nabla a(x)\times(G_j-\wid{E_j})]
=\nabla a\nabla\cdot(G_j-\wid{E_j})-(\nabla a\cdot\nabla)(G_j-\wid{E_j})\\
+[(G_j-\wid{E_j})\cdot\nabla]\nabla a-(G_j-\wid{E_j})\Delta a.
\enn
The estimates (\ref{3.20}) and (\ref{3.21}) imply that the third and forth terms on the right-hand side of
the above equation are uniformly bounded in $L^2(D_0)$ for $j\in\N$.
For the first and second terms, since $\nabla a|_{B_1}=0$, we only need to show that
$\nabla(G_j-\wid{E_j})$ and $\nabla\cdot(G_j-\wid{E_j})$ are uniformly bounded in $L^2(D_0\se\ov{B_1})$.
First, from (\ref{3.11+}) it is noted that
\ben
G_j-E^i_j=T_2G_j-T_1(G_j-E^i_j)-T_1E_j^i,
\enn
where $T_1,T_2$ are defined just after (\ref{3.11}).
Obviously, $T_2G_j\in W^{2,p}(D)\hookrightarrow H^1(D)$ and $T_1G_j^s\in H^1(D)$
since, by (\ref{3.21}), $G_j\in L^p(D)$ with $6/5\le p<3/2$ and $(G_j-E^i_j)\in L^2(D)$.
For $x\in D_0\se\ov{B_1}$, it is easy to see that $\|T_1E_j^i\|_{H^1(D\se \ov{B_1})}\leq C$.
Then we have
\ben
\|G_j-E^i_j\|_{H^1(D\se\ov{B_1})}\leq C.
\enn
This, together with the estimate $\|\wid{E_j^s}\|_{H^1(D_0)}\leq C$, implies that
$\nabla(G_j-\wid{E_j})$ and $\nabla\cdot(G_j-\wid{E_j})$ are uniformly bounded in $L^2(D_0\se\ov{B_1})$,
and then
\ben
\curl[\nabla a(x)\times(G_j-\wid{E_j})]\in L^2(D_0)
\enn
uniformly for $j\in\N$. Therefore, we have that $\|\curl\curl W(j)\|_{L^2(D_0)}\le C$,
so
\be\label{3.24}
\|(h_1(j),h_2(j))\|_{\tau(\pa D_0)}\le\|W(j)\|_{\mathcal{U}(D_0)}\le C,
\en
where $C$ is independent of $j\in\N$.

Similarly as in the acoustic case, $D_0$ can be chosen so that
$k^2<\min\{\la_1(D_0),\la_1(D_0)/\sup(n)\}$.
It then follows from (\ref{3.24}) and Lemma \ref{le3.1} that
$$
\|\wid{E_j}\|_{L^2(D_0)}=\|E_0(j)\|_{L^2(D_0)}\le C_1\|(h_1(j),h_2(j))\|_{\tau(\pa D_0)}\le C
$$
uniformly for $j\in\N.$ Therefore, we have
\ben
\|\curl(p\Phi(\cdot,z_j))\|_{L^2(D_0)}
-\|\wid{E_j^s}\|_{L^2(D_0)}\le\|\wid{E_j}\|_{L^2(D_0)}\le C
\enn
uniformly for $j\in\N$. This is a contradiction since
$\|\curl(p\Phi(\cdot,z_j))\|_{L^2(D_0)}\rightarrow\infty$ as $j\rightarrow\infty$
and $\|\wid{E_j^s}\|_{L^2(D_0)}$ is bounded uniformly for $j\in\N$.
We then have $D=\wid{D}$. The proof is thus completed.
\end{proof}

\begin{remark}\label{rm3.3}{\rm
If $n\in C^{2,\al}_0(\R^3)$ with $0<\al<1$, then it was proved in \cite{CP} that $n$ (and therefore $D$) 
can be uniquely determined by the electric far-field patterns $\wid{E}^\infty(\wi{x};d;p)$ for all 
$\wi{x},d\in\Sp^2$ and all polarizations $p\in\R^3$. However, the result in \cite{CP} is not applicable to
the case in this paper. 
}
\end{remark}

\section*{Acknowledgements}

The work was partly supported by NNSF of China under grants 91430102, 91630309, 11401568, 11501558,
the China Postdoctoral Science Foundation under grants 2015M580827 and 2016T90900
and the National Center for Mathematics and Interdisciplinary Sciences, CAS.


\end{document}